\documentclass{article}

\usepackage{amsmath}
\usepackage{amssymb}
\usepackage{amsthm}
\usepackage{amscd}

\newtheorem{thm}{Theorem}[section]
\newtheorem{prop}[thm]{Proposition}
\newtheorem{lem}[thm]{Lemma}
\newtheorem{cor}[thm]{Corollary}

\theoremstyle{definition}
\newtheorem{dfn}[thm]{Definition}

\theoremstyle{remark}
\newtheorem{rem}{Remark}
\newtheorem*{acknowledgments}{Acknowledgments}

\newcommand{\R}{\mathbb{R}}
\newcommand{\Z}{\mathbb{Z}}

\newcommand{\tr}{\mathrm{tr}}
\newcommand{\ch}{\mathrm{ch}}

\title{Mickelsson's twisted $K$-theory invariant and its generalizations}

\author{Kiyonori Gomi}

\date{}

\begin{document}

\maketitle

\begin{abstract}
Mickelsson's invariant is an invariant of certain odd twisted $K$-classes of compact oriented three dimensional manifolds. We reformulate the invariant as a natural homomorphism taking values in a quotient of the third cohomology, and provide a generalization taking values in a quotient of the fifth cohomology. These homomorphisms are related to the Atiyah-Hirzebruch spectral sequence. We also construct some characteristic classes for odd twisted $K$-theory in a similar vein.
\end{abstract}

\tableofcontents


\section{Introduction}
\label{sec:introduction}

In a word, twisted $K$-theory is topological $K$-theory with local coefficients. It was introduced by Donovan and Karoubi \cite{D-K}, and also by Rosenberg \cite{R}. Because of its applications to physics, twisted $K$-theory attracts much interest recently. In the context of string theory, twisted $K$-theory is thought of as natural home for Ramond-Ramond charges of D-branes with background $B$-fields. On compact Lie groups, twisted $K$-theory is related to the Verlinde ring through representation theory of loop groups \cite{FHT1}. 

\smallskip

Mickelsson's twisted $K$-theory invariant \cite{M} is an invariant of certain odd twisted $K$-classes of compact oriented $3$-dimensional manifolds. Though is given by integrating differential forms, the invariant detects \textit{some} torsion elements in the odd twisted $K$-groups, as opposed to the Chern character, which detects all non-torsion elements, but no torsion elements. For example, the twisted $K$-group of $SU(2)$ with non-trivial `local coefficients' is a torsion group, and twisted $K$-classes arising from positive energy representations of the loop group of $SU(2)$ are distinguished by Mickelsson's invariant. Notice that the twisted $K$-(co)homology of a simply connected simple Lie group is always a torsion group \cite{Bra,D} and its order has an interesting relation with the Verlinde ideal \cite{Bra}.

\medskip

The subjects of this paper are a reformulation and some generalizations of Mickelsson's invariant. In order to explain our reformulation of the invariant, we introduce $K_P^1(M)$, $h(P)$ and $\mu_1$ here: First, $K_P^1(M)$ denotes the odd twisted $K$-group of a manifold $M$, where $P$ is the datum of a `local system' twisting $K$-theory. Concretely, $P$ is a principal bundle over $M$ whose structure group is the projective unitary group $PU(H)$ of a separable infinite dimensional Hilbert space $H$. Second, $h(P) \in H^3(M, \Z)$ is a third integral cohomology class associated to $P$. Actually, this cohomology class, often called the Dixmier-Douady class, classifies principal $PU(H)$-bundles. Finally, $\mu_1 : K^1_P(M) \longrightarrow H^1(M, \Z)$ is a natural homomorphism. This is one of the simplest twisted $K$-theory invariant. For any element $x \in K^1_P(M)$, the cohomology class $\mu_1(x) \in H^1(M, \Z)$ can be realized as the homomorphism on the fundamental group $\mu_1(x) : \pi_1(M) \to \Z$ whose value at a loop $\ell : S^1 \to M$ is $\ell^*(x) \in K^1_{\ell^*P}(S^1) \cong K^1(S^1) \cong \Z$. 

Now, our reformulation of Mickelsson's invariant is as follows:

\begin{thm} \label{ithm:1}
For any principal $PU(H)$-bundle $P$ over a manifold $M$, there is a natural homomorphism
$$
\mu_3 : \ \mathrm{Ker}\mu_1 \longrightarrow
H^3(M, \Z)/(\mathrm{Tor} + h(P) \cup H^0(M, \Z)),
$$
where $\mathrm{Tor} \subset H^3(M, \Z)$ is the torsion subgroup, and $h(P) \cup H^0(M, \Z)$ is the subgroup consisting of the cup products of $h(P)$ and classes in $H^0(M, \Z)$.
\end{thm}

In the case that $M$ is compact, connected, oriented and $3$-dimensional, the orientation induces the isomorphism $H^3(M, \Z) \cong \Z$, by which we regard $h = h(P)$ as an integer. Then the homomorphism
$$
\mu_3 : \ \mathrm{Ker}\mu_1 \longrightarrow \Z/h\Z
$$
agrees with Mickelsson's invariant, up to a convention of sign.

\medskip

As is mentioned, Mickelsson's invariant is given by integrating differential forms: To be concrete, we represent an odd twisted $K$-class on $M$ by a section $g$ of the fiber bundle $P \times_{Ad} U_1(H)$ over $M$ associated to the conjugate action of $PU(H)$ on the group $U_1(H)$ consisting of unitary operators on $H$ which differ from the identity operator by trace class operators. The choice of a good open cover $\{ U_i \}$ of $M$ allows us to express the section $g$ by maps $g_i : U_i \to U_1(H)$, so that we get a closed $3$-form on each $U_i$:
$$
\frac{-1}{24\pi^2} \tr[(g^{-1}_idg_i)^3].
$$
Then original Mickelsson's invariant is defined by integrating these local $3$-forms and by adding correction terms. The correction terms are also given by integrating lower differential forms, which measure the incompatibility of the $3$-forms on overlaps of $U_i$. The ambiguity of a choice necessary to a correction term leads to the quotient by $h \Z$.

\medskip

Now, the idea of our reformulation is simple: The local $3$-forms and lower differential forms measuring the incompatibility constitute a \v{C}ech-de Rham $3$-cocycle. This $3$-cocycle represents a well-defined element in $H^3(M, \R)/(h(P) \cup H^0(M, \Z))$. With some knowledge about characteristic classes for twisted $K$-theory \cite{A-Se2}, we finally get the homomorphism $\mu_3$. An interesting point is that we encounter a \textit{smooth Deligne cocycle} \cite{Bry} along the way of this construction.

\medskip

We should remark that Mickelsson himself might have this idea of reformulation: a construction based on Quillen's superconnection is carried out on a compact Lie group in \cite{M-P}.

\medskip

It is natural to generalize the construction above by using the odd forms
$$
\frac{1}{(2\pi\sqrt{-1})^k} \frac{((k-1)!)^2}{(2k-1)!} 
\tr[(g_i^{-1}dg_i)^{2k-1}],
$$
which are the pull-back of the integral primitive forms on $U_1(H)$, (\cite{Si}). The case of $k = 1$ reproduces $\mu_1$ essentially. The case of $k = 3$ gives:

\begin{thm} \label{ithm:2}
For any principal $PU(H)$-bundle $P$ over a manifold $M$, there is a natural homomorphism
$$
\bar{\mu}^{\R}_5 : \ \mathrm{Ker}\mu_3 \longrightarrow
H^5(M, \R)/(h(P) \cup H^2(M, \R)).
$$
\end{thm}

Notice that $\bar{\mu}^{\R}_5$ takes values in the quotient by $h(P) \cup H^2(M, \R)$ rather than $h(P) \cup H^2(M, \Z)$. Therefore $\bar{\mu}^{\R}_5$ does not detect torsion elements. The essential reason is that $\mu_3$ is defined through differential forms, as will be seen in the proof of the well-definedness of $\bar{\mu}_5^{\R}$. To improve $\bar{\mu}^{\R}_5$ to detect torsions, an explicit description of $\mu_3$ in terms of integral cocycle seems to be required. 

\medskip

One may notice a similarity between maps $\mu_1$, $\mu_3$ and $\bar{\mu}^{\R}_5$ and the surjections in the short exact sequence
$$
0 \longrightarrow 
F^{2k+1}K^1_P(M) \longrightarrow 
F^{2k-1}K^1_P(M) \longrightarrow 
E_\infty^{2k-1, 0} \longrightarrow 
0
$$
in computing $K^1_P(M)$ through the Atiyah-Hirzebruch spectral sequence \cite{A-Se2}. Actually, $\mu_1$ agrees with the above surjection with $k = 1$. For $\mu_3$ and $\bar{\mu}^{\R}_5$, we can prove factorizations through the surjections, as will be detailed in Section \ref{sec:comarison_with_AHSS}. This result together with an example of $\mu_3$ reproves a computational result about twisted $K$-theory of $SU(3)$.

\medskip

Aside from the generalization above, we have other generalizations: The idea is the same, and this case considers \v{C}ech-de Rham cocycles based on
\begin{align*}
& \tr[(g_i^{-1}dg)^3] \tr[g^{-1}_idg_i], &
& \tr[(g_i^{-1}dg)^5]\tr[(g_i^{-1}dg)^3] \tr[g^{-1}_idg_i].
\end{align*}

\begin{thm} \label{ithm:3}
For any principal $PU(H)$-bundle $P$ over a manifold $M$, there are natural maps
\begin{align*}
\nu_4 &: \ K_P^1(M) \longrightarrow H^4(M, \R), &
\nu_9 &: \ K_P^1(M) \longrightarrow H^9(M, \R).
\end{align*}
\end{thm}

Generally, a characteristic class (or Chern class) for odd twisted $K$-theory means a natural map $K_P^1(M) \to H^n(M, \R)$ defined for any $P$. Characteristic classes for even twisted $K$-theory are studied in \cite{A-Se2}. Applying ideas in \cite{A-Se2}, we can enumerate all the possibility of characteristic classes for odd twisted $K$-theory (Lemma \ref{lem:char_class_real}). According to this result, $\nu_4$ and $\nu_9$ are bases of the spaces of characteristic classes with values in $H^4(M, \R)$ and $H^9(M, \R)$. 

\medskip

So far, characteristic classes for twisted $K$-theory are constructed by an algebro topological method \cite{A-Se2} or by a differential geometric method (see \cite{G-T1} for example). Our construction is in some sense an intermediate one.

\medskip

The Chern character for twisted $K$-theory also admits a number of formulations (for example \cite{A-Se2,CMW,FHT4,M-S}). Developing our construction, we may formulate the Chern character for odd twisted $K$-theory as a homomorphism with values in twisted \v{C}ech-de Rham cohomology. A partial result in Section \ref{sec:comarison_with_AHSS} justifies this, but the full formulation needs an efficient way to handle the \v{C}ech-cochains involving the odd forms $\tr[(g_i^{-1}dg_i)^{2k-1}]$.

\bigskip

The organization of the present paper is as follows: In Section \ref{sec:smooth_Deligne_cohomology}, we recall \v{C}ech cohomology and smooth Deligne cohomology, to fix our convention of notations. In Section \ref{sec:twisted_K_theory}, we review twisted $K$-theory. We also review here the classification of principal $PU(H)$-bundles, the Atiyah-Hirzebruch spectral sequence (AHSS), characteristic classes and the homomorphism $\mu_1$. Then, in Section \ref{sec:reformulation}, we give our reformulation of Mickelsson's invariant. We start with an introduction of a smooth Deligne $4$-cocycle and then give the \v{C}ech-de Rham $3$-cocycle giving $\mu_3$. We also provide examples of computations of $\mu_3$ at the end of this section. In Section \ref{sec:generalization}, we construct our generalizations $\bar{\mu}_5^{\R}$, $\nu_4$ and $\nu_9$. Then, in Section \ref{sec:comarison_with_AHSS}, we state and prove the factorizations of $\mu_3$ and $\bar{\mu}_5^{\R}$. We also present a part of a twisted \v{C}ech-de Rham cocycle suggesting a possible construction of the Chern character. Finally, in Section \ref{sec:cocycle_condition} and \ref{sec:well_definedness}, we provide formulae necessary for the proof of some statements in Section \ref{sec:reformulation} and \ref{sec:generalization} involving computations of \v{C}ech-de Rham cochains. These computations are simple but lengthy. Hence we separated them for the readability.

\medskip

\begin{acknowledgments}
I would like to thank Jouko Mickelsson for comments about an earlier version of the paper. The author's research is supported by the Grant-in-Aid for Young Scientists (B 23740051), JSPS.
\end{acknowledgments}



\section{Smooth Deligne cohomology}
\label{sec:smooth_Deligne_cohomology}

To fix our conventions, we review \v{C}ech and smooth Deligne cohomology.

\subsection{\v{C}ech cohomology}

Let $X$ be a topological space, and $\mathcal{F}$ a sheaf of abelian groups on $X$. For an open cover $\{ U_i \}$ of $X$, the \v{C}ech cohomology $H^*(\{ U_i \}, \mathcal{F})$ with coefficients in $\mathcal{F}$ is defined to be the cohomology of the cochain complex $(C^*(\{ U_i \}, \mathcal{F}), \delta)$. Here the group of $p$-cochains is given by $C^p(\{ U_i \}, \mathcal{F}) = \prod_{i_0 \ldots i_p} \mathcal{F}(U_{i_0 \cdots i_p})$, where $U_{i_0 \cdots i_p} = U_{i_0} \cap \cdots \cap U_{i_p}$ as usual. The coboundary operator $\delta : C^p(\{ U_i \}, \mathcal{F}) \to C^{p+1}(\{ U_i \}, \mathcal{F})$ is defined by
$$
(\delta c)_{i_0 \cdots i_{p+1}} 
= \sum_{j = 0}^{p+1} (-1)^j c_{i_0 \cdots \widehat{i_j} \cdots i_{p+1}}|_{U_{i_0 \cdots i_{p+1}}}.
$$
To eliminate the dependence on the open cover, we take the direct limit
$$
H^*(X, \mathcal{F}) = \varinjlim H^*(\{ U_i \}, \mathcal{F})
$$
with respect to the direct system consisting of open covers of $X$. In the case $X$ is a manifold, there exists a so-called \textit{good} open cover \cite{B-T}. For a good cover, we have the canonical isomorphism $H^*(\{ U_i \}, \mathcal{F}) \cong H^*(M, \mathcal{F})$.

\medskip

If $\mathcal{F}$ is a sheaf of rings, then we can introduce a multiplication to $C^*(\{ U_i \}, \mathcal{F})$ as follows. Let $a = (a_{i_0 \cdots i_p})$ and $b = ( b_{i_0 \cdots i_q} )$ are $p$- and $q$-cochains, respectively. Then their product $a \cup b = ( (a \cup b)_{i_0 \cdots i_{p+q}} )$ is the $(p+q)$-cochain defined by
$$
(a \cup b)_{i_0 \cdots i_{p+q}}
= a_{i_0 \cdots i_p} \cdot b_{i_{p+1} \cdots i_{p+q}}.
$$
We can easily verify $\delta (a \cup b) = (\delta a) \cup b + (-1)^p a \cup (\delta b)$, so that the multiplication makes the cohomology $H^*(\{ U_i \}, \mathcal{F})$ into a (graded) ring.

\medskip

For a sheaf of complex $\mathcal{F} = (\mathcal{F}^*, d)$, the \v{C}ech cohomology $H^*(\{ U_i \}, \mathcal{F})$ is defined as follows: We construct a double complex by setting $C^p(\{ U_i \}, \mathcal{F}^q) = \prod_{i_0 \cdots i_p} \mathcal{F}^q(U_{i_0 \cdots i_p})$. One coboundary operator is the \v{C}ech coboundary operator $\delta : C^p(\{ U_i \}, \mathcal{F}^q) \to C^{p+1}(\{ U_i \}, \mathcal{F}^q)$. The other coboundary operator $d : C^p(\{ U_i \}, \mathcal{F}^q) \to C^p(\{ U_i \}, \mathcal{F}^{q+1})$ is induced from that on the complex $(\mathcal{F}^*, d)$. We define the total complex $(C^*(\{ U_i \}, \mathcal{F}), D)$ to be
$$
C^r(\{ U_i \}, \mathcal{F}) 
= \bigoplus_{r = p + q} C^p(\{ U_i \}, \mathcal{F}^q).
$$
The total coboundary operator $D$ acts as $D = \delta + (-1)^pd$ on $C^p(\{ U_i \}, \mathcal{F}^q)$. The cohomology of this total complex is $H^*(\{ U_i \}, \mathcal{F})$. Taking the direct limit as before, we get $H^*(X, \mathcal{F})$. If $X$ is a manifold and $\{ U_i \}$ is a good cover, then $H^*(\{ U_i \}, \mathcal{F}) \cong H^*(X, \mathcal{F})$. 

\medskip

A typical example of the sheaf of complex is the de Rham complex $\Omega = (\Omega^*, d)$ on a manifold $X$. In this case, $H^*(\{ U_i \}, \Omega)$ and $H^*(X, \Omega)$ are called the \v{C}ech-de Rham cohomology. As is well-known \cite{B-T}, the cochain map
\begin{align*}
C^p(\{ U_i \}, \R) & \longrightarrow C^p(\{ U_i \}, \Omega), &
(c_{i_0 \cdots i_p}) &\mapsto (c_{i_0 \cdots i_p}, 0, \cdots, 0)
\end{align*}
induces an isomorphism $H^*(X, \R) \cong H^*(X, \Omega)$. Also, the \v{C}ech-de Rham cohomology and the de Rham cohomology $H^*(\Omega^*(M), d)$ are isomorphic via
\begin{align*}
\Omega^p(M) &\longrightarrow C^p(\{ U_i \}, \R), &
\omega &\mapsto (0, \cdots, 0, \omega|_{U_i}).
\end{align*}
These cochain maps give rise to ring isomorphisms, if we define the product $\alpha \wedge \beta$ of \v{C}ech-de Rham $m$- and $n$-cochains
\begin{align*}
\alpha &= 
(
\alpha^{[0]}_{i_0 \cdots i_{p-1}}, \ldots, 
\alpha^{[m-1]}_{i_0 \cdots i_{p - m}}
), & 
\beta &=
(
\beta^{[0]}_{i_0 \cdots i_{q-1}}, \ldots, 
\beta^{[n-1]}_{i_0 \cdots i_{q - n}}
)
\end{align*}
to be the following \v{C}ech-de Rham $(m+n)$-cochain:
$$
(\alpha^{[0]}_{i_0 \cdots i_m}\beta^{[0]}_{i_m \cdots i_{m+n}},
\cdots,
\sum_{\ell = 0}^k (-1)^{(m-k+\ell)\ell}
\alpha^{[m - k + \ell]}_{i_0 \cdots i_{k-\ell}}
\beta^{[n - \ell]}_{i_{k-\ell} \cdots i_k}, 
\cdots,
\alpha^{[m]}_{i_0}\beta^{[n]}_{i_0}).
$$

\subsection{Smooth Deligne cohomology}

For a smooth manifold $M$, the \textit{smooth Deligne complex} of order $n$ is the complex of sheaves
$$
\Z(n)_D^\infty : \
\Z \to 
{\Omega}^0 \overset{d}{\to} 
{\Omega}^1 \overset{d}{\to}
\cdots \overset{d}{\to}
{\Omega}^{n-1} \to
0 \to \cdots,
$$
where $\Z$ is the constant sheaf placed at degree $0$. The \textit{smooth Deligne cohomology} of $M$ is then defined to be $H^*(M, \Z(n)_D^\infty)$. 

\medskip

We have $H^*(M, \Z(0)_D^\infty) = H^*(M, \Z)$ clearly. For positive $n$, we have:

\begin{prop}[\cite{Bry}] \label{prop:Deligne_cohomology}
The following holds true for $n > 0$:
\begin{itemize}
\item[(a)]
If $p \neq n$, then we have
$$
H^p(M, \Z(n)_D^\infty) \cong
\left\{
\begin{array}{lc}
H^{p-1}(M, \R/\Z), & p < n, \\
H^p(M, \Z), & p > n.
\end{array}
\right.
$$

\item[(b)]
There are natural exact sequences
\begin{gather*}
0 \to
H^{n-1}(M, \R/\Z) \to
H^n(M, \Z(n)_D^\infty) \to 
\Omega^n(M)_{\Z} \to 
0, \\
0 \to 
\Omega^{n-1}(M)/\Omega^{n-1}(M)_{\Z} \to
H^n(M, \Z(n)_D^\infty) \to
H^n(M, \Z) \to
0,
\end{gather*}
where $\Omega^p(M)_{\Z} \subset \Omega^p(M)$ is the group of closed integral $p$-forms.
\end{itemize}
\end{prop}

We notice that the projections in the exact sequences in Proposition \ref{prop:Deligne_cohomology} have the following description in terms of \v{C}ech cocycles
\begin{align*}
H^n(M, \Z(n)_D^\infty) &\to 
\Omega^n(M)_{\Z}, &
(a_{i_0 \ldots i_n}, 
\alpha^{[0]}_{i_0 \ldots i_{n-1}}, \ldots,
\alpha^{[n-1]}_{i_0}) &\mapsto
d \alpha^{[n-1]}_{i_0}, \\
H^n(M, \Z(n)_D^\infty) &\to
H^n(M, \Z), &
(a_{i_0 \ldots i_n}, 
\alpha^{[0]}_{i_0 \ldots i_{n-1}}, \ldots,
\alpha^{[n-1]}_{i_0}) &\mapsto
(a_{i_0 \ldots i_n}).
\end{align*}

\medskip

It is known \cite{Bry} that the smooth Deligne complex admits a multiplication $\Z(m)_D^\infty \otimes Z(n)_D^\infty \to Z(m+n)_D^\infty$. The induced homomorphism on Deligne cochains can be described as follows: Let $p \le m$ and $q \le n$. For cochains
\begin{align*}
\check{\alpha} &= 
(
a_{i_0 \cdots i_p}, 
\alpha^{[0]}_{i_0 \cdots i_{p-1}}, \ldots, 
\alpha^{[m-1]}_{i_0 \cdots i_{p - m}}
), & 
\check{\beta} &=
(
b_{i_0 \cdots i_q}, 
\alpha^{[0]}_{i_0 \cdots i_{q-1}}, \ldots, 
\alpha^{[n-1]}_{i_0 \cdots i_{q - n}}
), 
\end{align*}
their product $\check{\alpha} \cup \check{\beta}$ is defined by
$$
(a \cup b, a \cup \beta^{[0]}, \cdots a \cup \beta^{[n-1]},
\alpha^{[0]} \cup d \beta^{[n-1]}, \cdots, 
\alpha^{[m-1]} \cup d \beta^{[n-1]}).
$$


\section{Twisted $K$-theory}
\label{sec:twisted_K_theory}

We here review some basics about twisted $K$-theory used in this paper.


\subsection{Principal $PU(H)$-bundle}
\label{subsec:projective_unitary_bundle}

To twist $K$-theory, we use a principal $PU(H)$-bundle in this paper. Here $PU(H) = U(H)/U(1)$ is the projective unitary group of a separable infinite dimensional Hilbert space $H$. As a topology of $PU(H)$, we consider that induced from the operator norm topology on $U(H)$.

\begin{prop}[\cite{Bry}]
The isomorphism classes of principal $PU(H)$-bundles on a manifold $M$ are in bijective correspondence with elements in $H^3(M, \Z)$. 
\end{prop}

We will write $h = h(P) \in H^3(M, \Z)$ for the element classifying a principal $PU(H)$-bundle $P \to M$. In \v{C}ech cohomology, $h(P)$ is described as follows: Let $\{ U_i \}$ be a good cover of $M$. If we choose local sections $s_i : U_i \to P|_{U_i}$, then we get the transition functions $\bar{\phi}_{ij} : U_{ij} \to PU(H)$ by $s_j = s_i\bar{\phi}_{ij}$. We also choose lifts $\phi_{ij} : U_i \to U(H)$ of $\bar{\phi}_{ij}$. As a convention, we always choose $\phi_{ij}$ so that $\phi_{ji} = \phi_{ij}^{-1}$. The choice of the lifts then defies $f_{ijk} : U_{ijk} \to U(1)$ by the formula
$$
\phi_{ij} \phi_{jk} = f_{ijk} \phi_{ik}.
$$
We further choose $\eta^{[0]}_{ijk} \in \Omega^0(U_{ijk})$ such that 
$$
f_{ijk} = \exp 2\pi i \eta^{[0]}_{ijk}.
$$ 
Now $h_{ijkl} = (\delta \eta^{[0]})_{ijkl} \in \Z$ defines a (Deligne) $3$-cocycle:
$$
h = (h_{i_0i_1i_2i_3}) \in \check{Z}^3(\{ U_i \}, \Z),
$$
which represents $h(P) \in H^3(M, \Z)$ classifying $P$. 

We remark that the Deligne $3$-cocycle $(h_{ijkl}, \eta^{[0]}_{ijk}) \in Z^3(\{ U_i \}, \Z(1)_D^\infty)$ also represents $h(P) \in H^3(M, \Z) \cong H^3(M, \Z(1)_D^\infty)$.

\medskip

So far, we only concerned with a topology on $PU(H)$, so that a principal $PU(H)$-bundle means a topological principal $PU(H)$-bundle. However, we need a `smooth' structure on a principal $PU(H)$-bundle in this paper. We have two possible methods to make sense of it: The first method is to introduce $PU(H)$ a structure of a manifold. This would be possible based on the fact that $U(H)$ gives rise to a Banach manifold \cite{Q}. The second method is to introduce a `smoothness' to a topological principal $PU(H)$-bundle $P$ over a manifold $M$ through its transition functions. Namely, let $\bar{\phi}_{ij} : U_{ij} \to PU(H)$ be transition functions of $P$ with respect to a good cover of $M$. Then, we say $P$ is smooth if there is a smooth lift $\phi_{ij} : U_{ij} \to U(H)$ of $\bar{\phi}_{ij}$.

Clearly, if $P$ is a smooth principal $PU(H)$-bundle in the first method, then so is in the second. In each formulation, the classification of smooth principal $PU(H)$-bundles is the same as in the case of the topological classification. The need for the smoothness of a principal $PU(H)$-bundle in this paper only enters through lifts of the transition functions. Hence the second method, which is easier to handle, is adequate for our purpose.


\subsection{Twisted $K$-theory}

For a principal $PU(H)$-bundle $P \to M$ and $i \in \Z$, the $P$-twisted $K$-groups $K_P^i(M)$ of $M$ are defined \cite{A-Se1,FHT1}. There are various realizations of $K^i(M)$. For the purpose of this paper, a particular construction of the odd twisted $K$-group $K^1_P(M)$ is necessary: Let $U_1(H)$ be the group of unitary operators on $H$ which differ from the identity operator $1$ by trace class operators:
$$
U_1(H) = \{ u \in U(H) |\ \mbox{$u - 1$ trace class operator} \}.
$$
The conjugate action of $PU(H)$ on $U_1(H)$ associates to $P$ the fiber bundle $P \times_{Ad} U_1(H)$. We denote by $\Gamma(M, P \times_{Ad} U_1(H))$ the space of sections. We say that $g_0, g_1 \in \Gamma(M, P \times_{Ad} U_1(H)$ are homotopic if there is $\tilde{g} \in \Gamma(M \times [0, 1], (P \times [0, 1]) \times_{Ad} U_1(H))$ such that $\tilde{g}|_{M \times \{ i \}} = g_i$ for $i = 0, 1$. The homotopy classes of sections of $P \times_{Ad} U_1(H)$ constitute the odd twisted $K$-group:
$$
K^1_P(M) = \Gamma(M, P \times_{Ad} U_1(H))/\mathrm{homotopy}.
$$ 
The group structure is induced from that on $U_1(H)$.

\medskip

The group $U_1(H)$ has a structure of a Banach Lie group \cite{Q}. Thus, in the case where $M$ is a manifold, we can also introduce a smoothness to $P \times_{Ad} U_1(H)$ according to the smoothness of $P$. Then, we have two options to consider continuous or smooth sections of $P \times_{Ad} U_1(H)$ to define the odd twisted $K$-group. But, the resulting groups are naturally isomorphic, because a continuous section $g$ can be approximated by a smooth section homotopic to $g$.


\subsection{The Atiyah-Hirzebruch spectral sequence}

The twisted $K$-groups constitute a certain generalised cohomology. This fact leads to the Atiyah-Hirzebruch spectral sequence(\cite{A-Se1,A-Se2}): The $E_2$-term is
$$
E_2^{p, q} = H^p(M, K^q(\mathrm{pt})).
$$
Because of the Bott periodicity $K_P^{i + 2k}(M) \cong K_P^i(M)$, we have $E_2^{p, q} = E_3^{p, q}$. The differential $d_3 : E_3^{p, 2q} \to E_3^{p+3, 2q-2}$ is specified in \cite{A-Se2}: $d_3 = Sq^3_{\Z} - h(P) \cup$. Then the spectral sequence converges to a graded quotient of $K^*_P(M)$ for compact $M$: Namely, there are filtrations
\begin{align*}
K_P^{0}(M) &= F^0K_P^{0}(M) \supset
F^2K_P^{0}(M) \supset 
F^4K_P^{0}(M) \supset \cdots, \\
K_P^{1}(M) &= F^1K_P^{1}(M) \supset
F^3K_P^{1}(M) \supset 
F^5K_P^{1}(M) \supset \cdots, 
\end{align*}
and the $E_\infty$-terms fit into the exact sequences
\begin{gather*}
0 \to F^{2q + 2}K_P^{0}(M)
\to F^{2q}K_P^{0}(M) \to
E_\infty^{2q, 0} \to 0, \\
0 \to F^{2q + 3}K_P^{1}(M)
\to F^{2q + 1}K_P^{1}(M) \to
E_\infty^{2q + 1, 0} \to 0.
\end{gather*}
To describe $F^pK_P^*(M)$, we regard $M$ as a CW complex. Then $F^pK_P^*(M)$ is defined to be kernel of the restriction $K_P^*(M) \to K_P^*(M_{< p})$, where $M_{< p} \subset M$ is the union of cells of dimension less than $p$.

\medskip

By the help of the Atiyah-Hirzebruch spectral sequence, we can easily compute the twisted $K$-groups of a compact connected oriented $3$-dimensional manifold $M$: Let $P \to M$ is a non-trivial principal $PU(H)$-bundle. We use the orientation to identify $h = h(P) \in H^3(M, \Z) \cong \Z$ with an integer. Because $Sq^3_{\Z}$ is trivial on $H^0(M, \Z)$, we have
\begin{align*}
E_\infty^{0, 0} &= 0, &
E_\infty^{1, 0} &= H^1(M, \Z), &
E_\infty^{2, 0} &= H^2(M, \Z), &
E_\infty^{3, 0} &= \Z/h.
\end{align*}
Because $H^1(M, \Z)$ is always torsion free, we get:
\begin{align*}
K^0_P(M) &\cong H^2(M, \Z), &
K^1_P(M) &\cong H^1(M, \Z) \oplus \Z/h.
\end{align*}


\subsection{Characteristic class}
\label{subsec:char_class}

For twisted $K$-theory of degree $i$, we mean by a \textit{characteristic class} or a \textit{Chern class} \cite{A-Se2} (with coefficient in $A$) a natural map
$$
c : \ K_P^i(M) \longrightarrow H^*(M, A)
$$
which is defined for any principal $PU(H)$-bundle $P$ over a manifold $M$. (We can consider more general spaces, such as CW complexes. But we restrict ourselves to consider manifolds only for simplicity.) For even twisted $K$-theory, such characteristic classes are known to be in one to one correspondence with cohomology classes of a certain space \cite{A-Se2}. A generalization of this result is that characteristic classes for odd twisted $K$-theory are in one to one correspondence with elements in $H^*(Y, A)$, where $Y = EPU(H) \times_{Ad} U_1(H)$ is the fiber bundle associated to the universal $PU(H)$-bundle $EPU(H) \to BPU(H)$ and the adjoint action of $PU(H)$ on $U_1(H)$.

\begin{lem} \label{lem:char_class_real}
Let $P(t) = \sum_{n \ge 0} \mathrm{dim}H^n(Y, \R) t^n$ be the generating function for the dimension of $H^n(Y, \R)$. Then we have
\begin{align*}
P(t)
&=
t^2 + t^3 + (1 - t^2) \prod_{i \ge 0} (1 + t^{2i+1}) \\
&=
1 + t + t^3 + t^4 + t^8 + t^9 + t^{12} 
+ t^{13} + t^{15} + 2t^{16} + t^{17} + \cdots.
\end{align*}
\end{lem}

\begin{proof}
The essential ideas are due to \cite{A-Se2}: We use the Serre spectral sequence:
$$
E_2^{p, q} = H^p(BPU(H), \R) \otimes H^q(U_1(H), \R) 
\Longrightarrow H^{p+q}(Y, \R).
$$
Recall that $U_1(H)$ is homotopy equivalent to $U(\infty) = \varinjlim_n U(n)$, as is well-known (\cite{Q} for example). Its real cohomology is the exterior algebra generated by elements $x_1, x_3, x_5, \ldots$ of degree $1, 3, 5, \ldots$ respectively: $H^*(U(n), \R) = \Lambda^*_{\R}(x_1, x_3, x_5, \ldots)$. On the other hand, $BPU(H)$ is homotopy equivalent to the Eilenberg-MacLane space $K(\Z, 3)$. Its real cohomology is the exterior algebra generated by a single generator $h$ of degree $3$. As a result, we have $E_3 = E_2$ and $E_\infty = E_4$. Now, by generalizing an argument in \cite{A-Se2}, we can describe the differential $d_3$ on generators $h$ and $x_i$ as follows:
\begin{align*}
d_3 h &= 0, &
d_3 x_1 &= 0, &
d_3 x_{2i+1} &= h x_{2i-1}, \ (i > 0).
\end{align*}
To compute the cohomology $d_3 : E_3^{0, *} \to E_3^{3, * - 2}$, we follow another idea in \cite{A-Se2}: Let $\Lambda = \Lambda_{\R}(x_1, x_3, x_5, \ldots)$ be the exterior algebra generated on elements $x_{2i+1}$ of degree $2i + 1$, ($ i \ge 0$) and define a derivation $d : \Lambda \to \Lambda$ by
\begin{align*}
d x_1 &= 0, &
d x_{2i+1} &= x_{2i-1}, \ (i > 0).
\end{align*}
Then we have
\begin{align*}
E^0_\infty &\cong \mathrm{Ker}(d), &
E^3_\infty &\cong \mathrm{Coker}(d) \oplus \R h.
\end{align*}
Let $\Lambda_k \subset \Lambda$ be the subspace consisting of homogeneous elements of degree $k$. In the same way as in \cite{A-Se2}, we can prove that $d : \Lambda_{n+2} \to \Lambda_n$ is surjective for all $n > 0$. Thus, for $n \ge 3$, we have the short exact sequence
$$
0 \longrightarrow
\mathrm{Ker}(d) \cap \Lambda_n \longrightarrow
\Lambda_n \overset{d}{\longrightarrow}
\Lambda_{n-2} \longrightarrow
0,
$$
so that $\mathrm{dim} (\mathrm{Ker}(d) \cap \Lambda_n ) = \mathrm{dim} \Lambda_n - \mathrm{dim}\Lambda_{n-2}$. The generating function of the dimension of $\Lambda_n$ is
$$
\sum_{n \ge 0} \mathrm{dim}\Lambda_n t^n 
= \prod_{i \ge 0} (1 + t^{2i + 1}).
$$
Assembling the results above, we get
$$
\sum_{n \ge 0} \mathrm{dim}(\mathrm{Ker}(d) \cap \Lambda_n ) t^n
= t^2 + (1 - t^2) \prod_{i \ge 0} (1 + t^{2i+1}).
$$
Taking $E^3_\infty = \R h$ into account, we complete the proof.
\end{proof}

\begin{lem} \label{lem:char_class}
We have
$$
\begin{array}{c|c|c|c|c|c}
p & 0 & 1 & 2 & 3 & \cdots \\
\hline
H^p(Y, \Z) & \Z & \Z & 0 & \Z & \cdots
\end{array}
$$
\end{lem}

\begin{proof}
As before, we use the Serre spectral sequence:
$$
E_2^{p, q} = H^p(BPU(H), \Z) \otimes H^q(U_1(H), \Z) 
\Longrightarrow H^{p+q}(Y, \Z).
$$
With basic knowledge of the cohomology of $U(n)$, we easily see that the integral cohomology of $U(\infty) \simeq U_1(H)$ is the exterior algebra generated by elements $z_{2i+1}$ of degree $2i+1$, ($i \ge 0$): $H^*(U_1(H), \Z) = \Lambda^*_{\Z}(z_1, z_3, z_5, \ldots)$. A computation of the cohomology of $K(\Z, 3) \simeq BPU(H)$ can be found in \cite{B-T}:
$$
\begin{array}{c|c|c|c|c|c|c|c|c}
p & 0 & 1 & 2 & 3 & 4 & 5 & 6 & \cdots \\
\hline
H^p(BPU(H), \Z) & \Z & 0 & 0 & \Z & 0 & 0 & \Z/2 & \cdots
\end{array}
$$
Now, we can readily compute the spectral sequence to identify $H^p(Y, \Z)$ for $p = 0, 1, 2$. The third cohomology fits into the exact sequence:
$$
0 \to
E^{3, 0}_{\infty} \to 
H^3(Y, \Z) \to
E^{0, 3}_{\infty} \to
0.
$$
In computing $E^{3, 0}_{\infty}$, the first possibly non-trivial differential is $d_3 : E^{0, 3}_3 \to E^{3, 1}_3$ with $E^{0, 3}_3 \cong \Z$ and $E^{3, 1}_3 \cong \Z$. In the case of real coefficients, $d_3$ above is non-trivial by Lemma \ref{lem:char_class_real}. This implies that $d_3$ with integral coefficients is also non-trivial, so that $E^{3, 0}_4 = 0$. Hence $E^{3, 0}_\infty = 0$ and $H^3(Y, \Z) \cong \Z$.
\end{proof}

\subsection{The determinant and the first cohomology}
\label{subsec:mu_1}

The following characteristic class is essentially given in \cite{M}.

\begin{prop} \label{prop:mu_1}
For any principal $PU(H)$-bundle $P$ over a manifold $M$, there is a natural homomorphism
$$
\mu_1 : \ K_P^1(M) \to H^1(M, \Z).
$$
\end{prop}

\begin{proof}
Because the determinant $\det : U_1(H) \to U(1)$ is invariant under the conjugation action of $U(H)$ on $U_1(H)$, we have a natural homomorphism
$$
\Gamma(M, P \times_{\mathrm{Ad}} U_1(H)) \to C^\infty(M, U(1)).
$$
This induces $\mu_1$, since the homotopy classes of $C^\infty(M, U(1))$ is $H^1(M, \Z)$.
\end{proof}

It is possible to describe the construction of $\mu_1$ above via Deligne cohomology: As in Subsection \ref{subsec:projective_unitary_bundle}, let $\{ U_i \}$ be a good cover of $M$ and $s_i : U_i \to P|_{U_i}$ local sections of $P$. We represent an element in $x \in K_P^1(M)$ by a section $g$ of $P \times_{\mathrm{Ad}} U_1(H)$. If we define $g_i : U_i \to U_1(H)$ by $g(x) = [s_i(x), g_i(x)]$, then $\phi_{ij}^{-1}g_i\phi_{ij} = g_j$, where $\phi_{ij} : U_{ij} \to U(H)$ are some lifts of the transition functions $\bar{\phi}_{ij} : U_{ij} \to PU(H)$. Because $\{ U_i \}$ is a good cover, we can find $\alpha^{[0]}_i \in \Omega^0(U_i)$ such that $\det(g_i) = \exp 2 \pi i \alpha^{[0]}_i$, which induces a $1$-cocycle $(a_{ij}) \in \check{Z}^1(\{ U_i \}, \Z)$ by setting $a_{ij} = (\delta \alpha^{[0]})_{ij}$. Summarizing, we get a Deligne $1$-cocycle
$$
\check{\alpha} = (a_{ij}, \alpha^{[0]}_i) 
\in \check{Z}^1(\{ U_i \}, \Z(1)_D^\infty).
$$
The cocycle conditions are:
\begin{align*}
(\delta \alpha^{[0]})_{ij} - a_{ij} &= 0, &
(\delta a)_{ijk} &= 0.
\end{align*}
The assignment $g \mapsto [(a_{ij}, \alpha^{[0]}_i)]$ gives a well-defined homomorphism
$$
\Gamma(M, P \times_{\mathrm{Ad}} U_1(H)) \longrightarrow 
H^1(M, \Z(1)_D^\infty) \cong C^\infty(M, U(1)),
$$
which is identified with the homomorphism in the proof of Proposition \ref{prop:mu_1}. Now, the $1$-cocycle $(a_{ij}) \in Z^1(\{ U_i \}, \Z)$ represents $\mu_1(x) = \mu_1([g])$.

\bigskip

We here identify the homomorphism $\mu_1$ with one appearing in the Atiyah-Hirzebruch spectral sequence. It is easy to see
$$
E_\infty^{1, 0} \subset \cdots \subset
E_5^{1, 0} = \mathrm{Ker}(d_3) \subset
E_3^{1, 0} = H^1(M, \Z).
$$
The $E_\infty$-term $E_\infty^{1, 0}$ fits into the exact sequence
$$
0 \to F^3K_P^{1}(M)
\to K_P^{1}(M) \to
E_\infty^{1, 0} \to 0.
$$

\begin{lem} \label{lem:identify_mu_1}
$\mu_1$ is identified with the composition of:
$$
K_P^{1}(M) \to
E_\infty^{1, 0} \to
E_5^{1, 0} \to
E_3^{1, 0} = H^1(M, \Z).
$$
Thus, for any $x \in K_P^{1}(M)$, we have $h(P) \cup \mu_1(x) = 0$ in $H^4(M, \Z)$.
\end{lem}

\begin{proof}
Tentatively, we write $\pi_1$ for the composition. Lemma \ref{lem:char_class} implies that, up to a constant multiple, there is a unique characteristic class for odd twisted $K$-theory taking its values in $H^1(M, \Z)$. As a generator of such characteristic classes, we can choose $\pi_1$. This is because $\pi_1$ induces an isomorphism $K^1(S^1) \cong H^1(S^1, \Z) \cong \Z$. Now, we can express $\mu_1$ as $\mu_1 = c \pi_1$, where $c \in \Z$ is an unknown constant. To specify $c$, it suffices to compute $\mu_1$ in the case of $M = S^1$. For the computation, we notice an inclusion $U(\infty) = \varinjlim_n U(n) \subset U_1(H)$ which induces a homotopy equivalence and is compatible with the determinant \cite{Q}. Then the inclusion $g : S^1 \to U(\infty)$ defines the element $x \in K^1(S^1)$ such that $\mu_1(x) = 1$ under the identification $H^1(S^1, \Z) \cong \Z$ induced from the standard orientation on $S^1$. On the other hand, we also have $\pi_1(x) = 1$. To see this, we regard $S^1$ as a CW complex consisting of a $0$-cell $e^0$ and a $1$-cell $e^1$. Then $\pi_1(x) \in H^1(S^1, \Z)$ is represented by a homomorphism $C_1(S^1) \to \Z$, where $C_1(S^1)$ is the cellular chain complex of degree $1$, which is generated by the $1$-cell $e^1$. From the construction of the Atiyah-Hirzebruch spectral sequence, the homomorphism $C_1(S^1) \to \Z$ associates to $e^1$ the element $[g] \in K^1(e^1, \partial e^1) \cong \Z$. Here the isomorphism $K^1(e^1, \partial e^1) \cong \Z$ is given by the definition $K^1(e^1, \partial e^1) \cong K^0(D^2, S^1)$ and the Thom isomorphism $K^0(D^2, S^1) \cong K^0(\mathrm{pt}) = \Z$. Then $g : S^1 \to U_1(H)$ corresponds to the triple $(E^0, E^1, f)$, where $E^0$ and $E^1$ are the trivial vector bundle on $D^2$ of rank $1$, and $f : E_0|_{S^1} \to E_1|_{S^1}$ is the vector bundle map that multiplies the fiber at $u \in S^1$ with $u \in S^1 = U(1)$. Such a triple represents the element in $1 \in K^0(D^2, S^1) \cong \Z$, so that $\pi_1([g]) = 1$. Consequently, $c = 1$ and $\mu_1 = \pi_1$. The remaining claim $h \cup \mu_1(x) = 0$ now follows from the fact that $d_3 : E_3^{1, 0} \to E_3^{4, -2}$ is identified with $d_3 = - h \cup$.
\end{proof}


\section{Mickelsson's invariant}
\label{sec:reformulation}

We here give our reformulation of Mickelsson's twisted $K$-theory invariant \cite{M}. Theorem \ref{ithm:1} in Section \ref{sec:introduction} is shown as Theorem \ref{thm:reformulation}. Some of the proof with involved computations are given in Section \ref{sec:cocycle_condition} and \ref{sec:well_definedness}.

\subsection{Deligne $4$-cocycle}

\begin{dfn}
Let $M$ be a smooth manifold.
\begin{itemize}
\item
For $f : M \to U_1(H)$ we define a $3$-form $C_3(f)$ on $M$ by
$$
C_3(f) 
= \tr [ (f^{-1}df)^3 ]
$$

\item
For $f, g : M \to U(H)$ such that $f$ or $g$ takes values in $U_1(H)$, we define a $2$-form $B_2(f, g)$ on $M$ by
$$
B_2(f, g)
= \tr [ f^{-1}df \wedge dg g^{-1} ].
$$
\end{itemize}
\end{dfn}

Note that $df$ in the definition of $C_3(f)$ is a $1$-form with values in trace class operators. Hence $(f^{-1}df)^3$ is a $3$-form with values in trace class operators, and we can take its trace. Also, in the definition of $B_2$, the $1$-form $f^{-1}df$ or $dg g^{-1}$ takes values in trace class operators. Since trace class operators form an ideal, we can take the trace of $f^{-1}df \wedge dg g^{-1}$.

\begin{dfn} \label{dfn:Deligne_4_cochain}
Let $P$ be any principal $PU(H)$-bundle over a manifold $M$. Suppose a section $g$ of $P \times_{\mathrm{Ad}} U_1(H)$ is given. We make the following choices as in Subsection \ref{subsec:projective_unitary_bundle} and \ref{subsec:mu_1}:
\begin{enumerate}
\item
a good cover $\{ U_i \}$ of $M$,

\item
local sections $s_i : U_i \to P|_{U_i}$ of $P$, 

\item
lifts $\phi_{ij} : U_i \to U(H)$ of the transition functions $\bar{\phi}_{ij}$,

\item
functions $\eta^{[0]}_{ijk} : U_{ijk} \to \R$ such that $\exp 2\pi i \eta^{[0]}_{ijk} = f_{ijk}$,

\item
functions $\alpha^{[0]}_i : U_i \to \R$ such that $\exp 2\pi i \alpha^{[0]}_i = \mathrm{det}(g_i)$.
\end{enumerate}
Then we define a Deligne $4$-cochain
$$
\check{\beta} = 
(
b_{ijklm}, 
\beta^{[0]}_{ijkl}, 
\beta^{[1]}_{ijk}, 
\beta^{[2]}_{ij},
\beta^{[3]}_i
)
\in \check{C}^4(\{ U_i \}, \Z(4)_D^\infty)
$$
as follows:
\begin{align*}
\beta^{[3]}_i
&
= - \frac{1}{24\pi^2} C_3(g_i) 
= - \frac{1}{24\pi^2} \tr [ (g_i^{-1}dg_i)^3 ], \\
\beta^{[2]}_{ij}
&=
- \frac{1}{8\pi^2} \{ B_2(g_j, \phi_{ji}) - B_2(\phi_{ji}, g_i) \} \\
&=
- \frac{1}{8\pi^2}
\tr[ 
g_j^{-1}dg_j \wedge d\phi_{ji}\phi_{ji}^{-1}
- \phi_{ji}^{-1}d\phi_{ji} \wedge dg_i g_i^{-1}
], \\
\beta^{[1]}_{ijk}
&=
- \eta^{[0]}_{ijk} d\alpha^{[0]}_k, \\
\beta^{[0]}_{ijkl}
&=
- h_{ijkl} \alpha^{[0]}_{\ell}, \\
b_{ijklm}
&=
- h_{ijkl} a_{lm}.
\end{align*}
\end{dfn}

We remark that $- \beta^{[2]}_{ij}$ essentially coincides with (1.6) in \cite{M}, since
$$
-\beta^{[2]}_{ij}
=
\frac{1}{8\pi^2} \tr [
d\phi_{ij}\phi_{ij}^{-1} 
(dg_ig_i^{-1} + g_i^{-1}dg_i 
+ g_id\phi_{ij}\phi_{ij}^{-1}g_i^{-1}- d\phi_{ij}\phi_{ij}^{-1})].
$$

\begin{lem} \label{lem:cocycle_condition_beta}
The following holds true:
\begin{itemize}
\item[(a)]
The Deligne $4$-cochain $\check{\beta}$ is a cocycle.

\item[(b)]
The assignment $g \mapsto \check{\beta}$ induces a well-defined homomorphism:
$$
\Gamma(M, P \times_{\mathrm{Ad}} U_1(H)) \longrightarrow 
H^4(M, \Z(4)_D^\infty).
$$

\end{itemize}
\end{lem}

\begin{proof}
For (a), the proof of $D \check{\beta} = 0$ will be given in Subsection \ref{subsec:coboundary}. For (b), by Lemma \ref{lem:change_alpha_in_beta}, \ref{lem:change_eta_in_beta}, \ref{lem:change_phi_in_beta} and \ref{lem:change_s_in_beta}, if we change the choices $\alpha^{[0]}_i$, $\eta^{[0]}_i$, $\phi_{ij}$ and $s_i$, then $\check{\beta}$ changes by a coboundary, so that we get a well-defined map
$$
\Gamma(M, P \times_{\mathrm{Ad}} U_1(H)) \longrightarrow
H^4( \{ U_i \}, \Z(4)_D^\infty).
$$
We can easily see that this map is compatible with the homomorphism of Deligne cohomology induced from a refinement of the open cover $\{ U_i \}$. Therefore the map in question is well-defined. Finally, Lemma \ref{lem:additivity_beta} implies that this map gives rise to a homomorphism. 
\end{proof}

\begin{lem} 
The homomorphism $\Gamma(M, P \times_{\mathrm{Ad}} U_1(H)) \to H^4(M, \Z(4)_D^\infty)$ in Lemma \ref{lem:cocycle_condition_beta} induces a well-defined homomorphism
$$
\mu_3^D : \ K_P^{1}(M) \to H^3(M, \R)/H^3(M, \Z).
$$
\end{lem}

\begin{proof}
Recall the two exact sequences of Deligne cohomology:
\begin{gather*}
0 \to H^3(M, \R/\Z) \to H^4(M, \Z(4)_D^\infty)
\to \Omega^4(M)_\Z \to 0, \\
0 \to \Omega^3(M)/\Omega^3(M)_\Z \to
H^4(M, \Z(4)_D^\infty) \to
H^4(M, \Z) \to 0.
\end{gather*}
By construction, we have $d \beta^{[3]} = 0$. Thus, under the surjection in the first exact sequence, the Deligne cohomology class associated to $g$ is mapped to the trivial $4$-form. On the other hand, under the surjection in the second exact sequence, the Deligne cohomology class is mapped to $- h \cup \mu_1([g]) \in H^4(M, \Z)$, which is trivial by Lemma \ref{lem:identify_mu_1}. Combining these facts, we conclude that the homomorphism in Lemma \ref{lem:cocycle_condition_beta} is identified with 
$$
\Gamma(M, P \times_{\mathrm{Ad}} U_1(H)) \longrightarrow
H^3(M, \R)/H^3(M, \Z).
$$
This descends to give $\mu_3^D$: since $H^3(M, \R)/H^3(M, \Z)$ is homotopy invariant, the homomorphism carries homotopic sections to the same element.
\end{proof}

\begin{lem} \label{lem:triviality}
$\mu_3^D$ is trivial.
\end{lem}

\begin{proof}
We have the natural exact sequence
$$
\cdots \to 
H^3(M, \Z) \to
H^3(M, \R) \overset{q}{\to}
H^3(M, \R/\Z) \to
\cdots.
$$
The composition of $\mu_3^D$ and the inclusion induced from the sequence above
$$
H^3(M, \R)/H^3(M, \Z) \to H^3(M, \R/\Z)
$$
defines a characteristic class $\bar{\mu} : K^1_P(M) \to H^3(M, \R/\Z)$ for odd twisted $K$-theory. 
As in Subsection \ref{subsec:char_class}, the characteristic class corresponds bijectively to an element $\bar{u} \in H^3(Y, \R/\Z)$, where $Y = EPU(H) \times_{Ad} U_1(H)$. More precisely, $\bar{\mu}$ corresponds to an element
$$
\bar{u} \in 
\mathrm{Im}[q:  H^3(Y, \R) \to H^3(Y, \R/\Z) ] \subset H^3(Y, \R/\Z).
$$
Lemma \ref{lem:char_class_real} and Lemma \ref{lem:char_class} tell us $H^3(Y, \R) \cong \R$ and $H^3(Y, \Z) \cong \Z$. We let $x \in H^3(Y, \Z)$ be a generator, and $x_{\R} \in H^3(Y, \R) \cong \R$ its real image. Then we can express $\bar{u}$ as $\bar{u} = q(r x_{\R})$ for some $r \in \R$. To specify $r$, it suffices to compute $\mu_3^D$ in the case where $M = S^3$ and $P$ is trivial. In this case, $K^1(S^3) \cong \Z$ is generated by the composition $g$ of inclusions $S^3 = SU(2) \subset U(2) \subset U(\infty) \subset U_1(H)$. For this $g$, we have $\check{\beta} = (0, 0, 0, 0, \beta^{[3]}_i)$ with $\beta^{[3]}_i$ the restriction of the closed integral $3$-form
$$
\frac{-1}{24\pi^2} \tr[(g^{-1}dg)^3] \in \Omega^3(SU(2))_{\Z}.
$$
Therefore $\mu_3^D([g]) = 0$ for $[g] \in K^1(S^3)$ and we find $r$ is an integer: $r \in \Z$. This implies $\bar{u} = 0$ and $\bar{\mu} = 0$ as well as the general triviality of $\mu_3^D$.
\end{proof}

\subsection{Reformulation of Mickelsson's invariant}

\begin{lem} \label{lem:Cech_deRham_3_cocycle}
Let $g$ be a section of $P \times_{Ad} U_1(H)$ given, and $\{ U_i \}, s_i, \phi_{ij}, \eta^{[0]}_{ijk}, \alpha^{[0]}_i$ the choices in Definition \ref{dfn:Deligne_4_cochain}. We suppose that $\mu_1([g]) = 0$, so that there is $\alpha : M \to \R$ such that $\alpha^{[0]}_i = \alpha|_{U_i}$.

\begin{itemize}
\item[(a)]
The following is a \v{C}ech-de Rham $3$-cocycle:
$$
\beta =
(
\beta^{[0]}_{ijkl}, 
\beta^{[1]}_{ijk}, 
\beta^{[2]}_{ij},
\beta^{[3]}_i
)
\in \check{Z}^3(\{ U_i \}, \Omega).
$$

\item[(b)]
The assignment $g \mapsto [\beta]$ induces a well-defined homomorphism
$$
\{ g \in \Gamma(M, P \times_{Ad} U_1(H)) |\ \mu_1([g]) = 0 \}
\to
H^3(M, \R)/(h(P) \cup H^0(M, \Z)).
$$
\end{itemize}
\end{lem}

\begin{proof}
For (a), we have $a_{ij} = 0$ by the assumption. Since $\check{\beta}$ is a cocycle, so is $\beta$. For (b), it suffices to show that the changes of the choices $s_i$, $\phi_{ij}$, $\eta^{[0]}_{ijk}$ and $\alpha$ result in changes of $\beta$ by a $3$-coboundary or a $3$-cocycle representing an element in $h(P) \cup H^0(M, \Z)$. For simplicity, we assume $M$ is connected. If $\alpha'$ is the other choice, then the difference $s = \alpha' - \alpha$ is an integer. Let $\beta'$ and $\beta$ denote the \v{C}ech-de Rham $3$-cocycles defined by using $\alpha'$ and $\alpha$, respectively, with the other choices unchanged. Then Lemma \ref{lem:change_alpha_in_beta} gives
$$
\beta' - \beta = (- h_{ijkl} s, 0, 0, 0),
$$
which represents an element in $h(P) \cup H^0(M, \Z)$. We then consider to change the choices of $\eta^{[0]}_{ijk}$, $\phi_{ij}$ and $s_i$. By Lemma \ref{lem:change_eta_in_beta}, \ref{lem:change_phi_in_beta} and \ref{lem:change_s_in_beta}, we easily see $\beta$ changes by a $3$-coboundary, so that (b) is proved.
\end{proof}

\begin{lem} \label{lem:reformulate_Mickelsson_invariant}
The homomorphism in Lemma \ref{lem:Cech_deRham_3_cocycle} induces a homomorphism
$$
\mu_3^{\R} : \ \mathrm{Ker}\mu_1 \to 
H^3(M, \R)/(h(P) \cup H^0(M, \Z))
$$
which is natural and makes the following diagram commutative:
$$
\begin{CD}
\mathrm{Ker}\mu_1 @>>>
K_P^{1}(M)  \\
@V{\mu_3^{\R}}VV @VV{\mu_3^D}V @.  \\
H^3(M, \R)/(h(P) \cup H^0(M, \Z)) @>>>
H^3(M, \R)/H^3(M, \Z).
\end{CD}
$$
\end{lem}

\begin{proof}
As in the case of $\mu_3^D$, we easily see that the homomorphism in Lemma \ref{lem:Cech_deRham_3_cocycle} descends to give $\mu_3^{\R}$. Now, the commutativity of the diagram is clear from the constructions of $\mu_3^{\R}$ and $\mu_3^D$.
\end{proof}

The $\mu_3^{\R}$ recovers original Mickelsson's invariant if $M$ is a compact connected oriented $3$-manifold. But, our reformulation is a slight refinement of $\mu_3^{\R}$. The following theorem leads to Theorem \ref{ithm:1} in Section \ref{sec:introduction}:

\begin{thm} \label{thm:reformulation}
For any principal $PU(H)$-bundle $P$ over a manifold $M$, there is a natural homomorphism
$$
\mu_3 : \ \mathrm{Ker}\mu_1 \longrightarrow
H^3(M, \Z)/(\mathrm{Tor} + h(P) \cup H^0(M, \Z))
$$
such that:
\begin{itemize}
\item[(a)]
$\mu_3^{\R}$ factors through $\mu_3$ as follows:
$$
\mathrm{Ker}\mu_1 \overset{\mu_3}{\to}
H^3(M, \R)/(\mathrm{Tor} + h \cup H^0(M, \Z)) \to
H^3(M, \R)/(h \cup H^0(M, \Z)).
$$

\item[(b)]
If $M$ is a compact connected oriented $3$-manifold, then $-\mu_3$ agrees with original Mickelsson's twisted $K$-theory invariant.
\end{itemize}
\end{thm}

\begin{proof}
The exact sequence:
$$
\begin{array}{ccccc}
\frac{H^3(M, \Z)}{\mathrm{Tor} + h \cup H^0(M, \Z)} & \to &
\frac{H^3(M, \R)}{h \cup H^0(M, \Z)} & \to &
\frac{H^3(M, \R)}{H^3(M, \Z)} \to 0
\end{array}
$$
together with Lemma \ref{lem:triviality} and \ref{lem:reformulate_Mickelsson_invariant} proves (a). Then (b) follows from the construction, since the integration formula in \cite{M} gives the isomorphism between the \v{C}ech-de Rham cohomology $H^3(\{ U_i \}, \Omega)$ and $\R$. (cf. \cite{G-T2})
\end{proof}

\subsection{Example}
\label{subsec:example}

We here present examples of computations of $\mu_3$. To have an element in odd twisted $K$-groups, we utilize the notion of the push-forward \cite{C-W,FHT1}. 

\smallskip

Let $M$ be a compact oriented connected $3$-dimensional manifold, $\mathrm{pt} \in M$ a point in $M$, and $i : \mathrm{pt} \to M$ the inclusion map. We use the orientation of $M$ to have the isomorphism $H^3(M, \Z) \cong \Z$, and regard $h = h(P)$ as an integer for a principal $PU(H)$-bundle $P \to M$ given. In this case, the normal bundle of $\mathrm{pt} \in M$ is trivial, so that the push-forward along $i$ is a homomorphism
$$
i_* : \ K^0(\mathrm{pt}) \longrightarrow K^1_P(M).
$$
Notice that an isomorphism $K^0_{P|_{\mathrm{pt}}}(\mathrm{pt}) \cong K^0(\mathrm{pt})$ is induced from a trivialization of $P|_{\mathrm{pt}}$, which is essentially unique in the present case, since $H^2(\mathrm{pt}, \Z) = 0$.

\begin{prop} \label{prop:exmaple_three_dim}
Let $M$ be a compact oriented connected $3$-dimensional manifold, and $P$ be a principal $PU(H)$-bundle $P$ over $M$. 
\begin{itemize}
\item[(a)]
The image of $i_*$ is in $\mathrm{Ker}\mu_1$.

\item[(b)]
For $1 \in K^0(\mathrm{pt}) \cong \Z$, we have $\mu_3(i_*(1)) = \bar{1} \in \Z/h$.
\end{itemize}
In the above, we write $\bar{n} \in \Z/h$ for the class of $n \in \Z$.
\end{prop}

\begin{proof}
First of all, we recall that the push-forward $i_*$ is the composition of: 
$$
K^0_{i^*P}(\mathrm{pt}) \to
K^1_{P|_D}(D, \partial D) \to
K^1_P(M, \overline{M \backslash D}) \to
K^1_P(M).
$$
In the above, $D \subset M$ is a $3$-dimensional disk containing $\mathrm{pt}$, playing the role of a disk bundle of the normal bundle of $\mathrm{pt} \subset M$. Using a trivialization of $P|_D$, we have $K^1_{P|_D}(D, \partial D) \cong K^1(D, \partial D)$. Notice that the Thom class of the trivial bundle $\R^3 \to \mathrm{pt}$ defines a generator $\tau \in K^1(D, \partial D) \cong \Z$. Then the first homomorphism $K^0_{i^*P}(\mathrm{pt}) \to K^1_{P|_D}(D, \partial D)$ is given by $1 \mapsto \tau$. The second homomorphism is the excision isomorphism, where $\overline{M \backslash D}$ stands for the closure of the complement of $D$ in $M$. Finally, the third homomorphism is to `forget' about the information of supports. Now, recall that $K^1(S^3)$ is generated by the composition of the inclusions $S^3 = SU(2) \subset U(\infty) \subset U_1(H)$. In view of the isomorphism $K^1(D, \partial D) \cong K^1(S^3)$, we can represent $\tau \in K^1(D, \partial D) \cong K^1(S^3)$ by a map $g_D : D \to U_1(H)$ such that $g|_{\partial D} = 1$ and
$$
\int_D \frac{-1}{24\pi^2} \tr[ (g_D^{-1}dg_D)^3 ] = 1.
$$
Using $g_D$, we can realize $i_*(1)$ by a section $g \in \Gamma(M, P \times_{Ad} U_1(H))$ as follows: Let $\{ U_i \}$ be a good cover of $M$. Since $M$ is compact, we can assume that the open cover is finite. Reducing the number of the open cover if necessary, we can find an open set $U_0$ and a point $\mathrm{pt} \in U_0$ such that $\mathrm{pt} \not\in U_i$ for $i \neq 0$. We can assume this is the original point in question. Let $D \subset U_0$ be a $3$-dimensional disk such that $\mathrm{pt} \in D \not\subset U_i$ for $i \neq 0$. For local sections $s_i : U_i \to P|_{U_i}$, we define $g_0 : U_0 \to U_1(H)$ to be the extension of $g_D$ by the identity operator $1 \in U_1(H)$, and $g_i$ to be the constant map $g_i = 1$ for $i \neq 0$. The section $g$ corresponding to $\{ g_i \}$ is a representative of $i_*(1)$. Then, for (a), it suffices to prove $\pi_1(\ell^*[g]) = 0$ for loops representing elements in $\pi_1(M)$. Because we can choose $D$ small so that $\ell(S^1) \cap D = \emptyset$, we have $\ell^* \mathrm{det}(g) = 1$ and hence $\mu_1([\ell^*g]) = 0$. For (b), observe that the \v{C}ech-de Rham cocycle $\beta$ representing $\mu_3([g])$ is of the form $\beta = (0, 0, 0, \beta^{[3]}_i)$. Further, $\beta^{[3]}_i = 0$ for $i \neq 0$ and $\int_{U_0} \beta^{[3]}_0 = 1$. Such a cocycle represents $1 \in H^3(M, \R)$ and hence $\mu_3([g]) = \bar{1}$.
\end{proof}

\smallskip

Next, we consider $\mu_3$ on $SU(3)$. Let $W \subset SU(3)$ be the subspace consisting of symmetric elements: $W = \{ U \in SU(3) |\ U^T = U \}$, where $U^T$ is the transpose of $U$. The map $U \mapsto U U^T$ induces the diffeomorphism $SU(3)/SO(3) \to W$. The Poincar\'{e} dual $\mathrm{PD}(W) \in H^3(SU(3), \Z) \cong \Z$ of $W \subset SU(3)$ is a generator. The characteristic class $W_3(N) \in H^3(W, \Z) \cong \Z/2$ of the normal bundle $N \to W$ of the inclusion $i : W \to SU(3)$ is non-trivial, and $i^*\mathrm{PD}(W) = W_3(N)$ holds true. (See \cite{MMS} for these facts.) For any principal $PU(H)$-bundle $P$ on $SU(3)$, we again regard $h = h(P)$ as an integer. Now, in the case that $h$ is odd, the push-forward along $i$ gives a homomorphism
$$
i_* : \ K^0(W) \longrightarrow K^1_P(SU(3)).
$$
It is known \cite{Min} that $K^0(W) \cong \Z$. Thus, the trivial vector bundle constitute $K^0(W)$. Note that $K^1_P(SU(3)) = \mathrm{Ker}\mu_1$ since $\pi_1(SU(3)) = 0$.

\begin{prop} \label{prop:example_SU(3)_odd}
Let $P \to SU(3)$ be a principal $PU(H)$-bundle such that $h = h(P) \in H^3(SU(3), \Z) \cong \Z$ is odd. For $1 \in K^0(W) \cong \Z$, we have $\mu_3(i_*(1)) = \bar{1} \in \Z/h$.
\end{prop}

\begin{proof}
As in \cite{MMS}, we can find an inclusion $j : S^3 \to SU(3)$ such that $S^3 \cap W = \mathrm{pt}$. Then $j^*$ induces an isomorphism $H^3(SU(3), \Z) \cong H^3(S^3, \Z)$. We here consider the following diagram:
$$
\begin{CD}
K^0(W) @>{i_*}>> K^1_P(SU(3)) @>{\mu_3}>> \Z/h \\
@VVV @VV{j^*}V @| \\
K^0(\mathrm{pt}) @>>> K^1_{j^*P}(S^3) @>{\mu_3}>> \Z/h,
\end{CD}
$$
where $K^0(\mathrm{pt}) \to K^1_{j^*P}(S^3)$ is the push-forward along the inclusion $\mathrm{pt} \to S^3$, and $K^0(W) \to K^0(\mathrm{pt})$ is the pull-back under the inclusion $\mathrm{pt} = W \cap S^3 \to W$. We can prove that the diagram is commutative: As in the case of $\mathrm{pt} \to S^3$, the push-forward $i_*$ is the composition of
$$
K^0(W) \to 
K^1_{P|_D}(D, S) \to 
K^1_P(SU(3), \overline{SU(3) \backslash D}) \to 
K_P^1(SU(3)),
$$
where $D$ and $S$ are the disk bundle and the sphere bundle of $N$. To understand $i_*$, we use an idea in \cite{H} to construct a gerbe classified by $\mathrm{PD}(W)$. We then use the description of twisted $K$-theory in \cite{FHT1} to construct a representative of the Thom class in $K^1_{P|_D}(D, S)$. Such a representative can be constructed so that its support concentrates at $W \subset D$. Accordingly, the restriction $j^*(i_*(1))$ admits the same description as the image of $1 \in K^0(\mathrm{pt})$ under the push-forward along $\mathrm{pt} \to S^3$ given in Proposition \ref{prop:exmaple_three_dim}. Once the commutativity of the diagram is established in this way, Proposition \ref{prop:exmaple_three_dim} completes the proof.
\end{proof}

In the case that $h$ is even, the push-forward along $i$ gives a homomorphism
$$
i_* : \ K^0_Q(W) \longrightarrow K^1_P(SU(3)),
$$
where $Q$ is a principal $PU(H)$-bundle such that $h(Q) = W_3(N)$. The Atiyah-Hirzebruch spectral sequence shows $K^0_Q(W) \cong 2 \Z$.

\begin{prop} \label{prop:example_SU(3)_even}
Let $P \to SU(3)$ be a principal $PU(H)$-bundle such that $h = h(P) \in H^3(SU(3), \Z) \cong \Z$ is even. For $2 \in K^0_Q(W) \cong 2\Z$, we have $\mu_3(i_*(2)) = \bar{2} \in \Z/h$.
\end{prop}

\begin{proof}
In this case, we consider the following diagram:
$$
\begin{CD}
K^0_Q(W) @>{i_*}>> K^1_P(SU(3)) @>{\mu_3}>> \Z/h \\
@VVV @VV{j^*}V @| \\
K^0(\mathrm{pt}) @>>> K^1_{j^*P}(S^3) @>{\mu_3}>> \Z/h,
\end{CD}
$$
where $K^0(\mathrm{pt}) \to K^1_{j^*P}(S^3)$ is the push-forward along the inclusion $\mathrm{pt} \to S^3$, and $K^0_Q(W) \to K^0(\mathrm{pt})$ is the pull-back under the inclusion $\mathrm{pt} \to W$. Since $h(Q) = W_3(N)$ is a torsion element, we can represent an element in $K^0_Q(W)$ by a twisted vector bundle or a bundle gerbe $K$-module \cite{BCMMS}. By the help of this description, we can prove the commutativity of the diagram along the same line as in Proposition \ref{prop:example_SU(3)_odd}. Now, notice that we can construct bundle gerbe $K$-modules representing elements in $K^0_Q(W)$ from representations of $\mathrm{Spin}^c(3) = U(2)$ whose center $U(1)$ act by the defining representation. Such representations are even dimensional, so that the virtual rank $K^0_Q(W) \to \Z$ induces the isomorphism $K^0_Q(W) \cong 2\Z$. Therefore the restriction $K^0_Q(W) \to K^0(\mathrm{pt})$ is the inclusion $2\Z \to \Z$, and Proposition \ref{prop:exmaple_three_dim} completes the proof. 
\end{proof}


\section{Generalization}
\label{sec:generalization}

We here generalize the construction in the previous section. Theorem \ref{ithm:2} in Section \ref{sec:introduction} is shown as Theorem \ref{thm:mu_5}, while Theorem \ref{ithm:3} is separated into Theorem \ref{thm:nu_4} and \ref{thm:nu_9}. Some details of the proof are given in Section \ref{sec:cocycle_condition} and \ref{sec:well_definedness}.

\subsection{Deligne $6$-cochain}

\begin{dfn} \label{dfn:group_5_cochains}
Let $M$ be a smooth manifold.
\begin{itemize}
\item
For $f : M \to U_1(H)$ we define a $5$-form $C_5(f)$ on $M$ by
$$
C_5(g) 
= \tr [ (f^{-1}df)^5 ].
$$

\item
For $f, g : M \to U(H)$ such that $f$ or $g$ takes values in $U_1(H)$, we define a $4$-form $B_4(f, g)$ on $M$ by
$$
B_4(f, g)
= \tr [ 
(f^{-1}df)(dgg^{-1})^3
+ \frac{1}{2} (f^{-1}df dgg^{-1})^2
+ (f^{-1}df)^3(dgg^{-1}) 
].
$$

\item
For $f, g, h : M \to U(H)$ such that $f$, $g$ or $h$ takes values in $U_1(H)$, we define a $3$-form $A(f, g, h)$ on $M$ by
$$
A(f, g, h)
=
\tr[ 
f^{-1}df g dhh^{-1} dg^{-1}
+ f^{-1}df dg dhh^{-1} g^{-1}
].
$$
\end{itemize}
\end{dfn}

\begin{dfn} \label{dfn:Deligne_6_cochain}
Let $P$ be any principal $PU(H)$-bundle over a manifold $M$, and $g$ a section of $P \times_{\mathrm{Ad}} U_1(H)$. Choosing $\{ U_i \}, s_i, \phi_{ij}, \eta^{[0]}_{ijk}$ and $\alpha^{[0]}_i$ as in Definition \ref{dfn:Deligne_4_cochain}, we define a Deligne $6$-cochain $\check{\gamma} \in \check{C}^6(\{ U_i \}, \Z(6)_D^\infty)$:
$$
\check{\gamma} = 
(
c_{i_0i_1i_2i_3i_4i_5i_6}, 
\gamma^{[0]}_{i_0i_1i_2i_3i_4i_5}, 
\gamma^{[1]}_{i_0i_1i_2i_3i_4}, 
\gamma^{[2]}_{i_0i_1i_2i_3},
\gamma^{[3]}_{i_0i_1i_2},
\gamma^{[4]}_{i_0i_1},
\gamma^{[5]}_{i_0}
)
$$
as follows:
\begin{align*}
\gamma^{[5]}_0
&
= \frac{i}{240\pi^3} C_5(g_0) 
= \frac{i}{240\pi^3} \tr [ (g_0^{-1}dg_0)^5 ], \\
\gamma^{[4]}_{01}
&=
\frac{i}{48\pi^3} 
\{ 
B_4(g_{1}, \phi_{10}) - B_2(\phi_{10}, g_{0}) 
\}, \\
\gamma^{[3]}_{012}
&=
- 2 \eta^{[0]}_{012} \beta^{[3]}_{2}
- d\eta^{[0]}_{012} \wedge \beta^{[2]}_{02} \\
&
+ d\eta^{[0]}_{012} \wedge \frac{1}{24\pi^2}
\{ 
B_2(\phi_{21}\phi_{10}, g_{0}) 
+ B_2(g_{2}, \phi_{21}\phi_{10}) 
\} \\
&
+ \frac{i}{48\pi^3}
\{ 
A(g_{2}, \phi_{21}, \phi_{10}) 
- A(\phi_{21}, g_{1}, \phi_{10})
+ A(\phi_{21}, \phi_{10}, g_{0}) 
\}, \\
\gamma^{[2]}_{0123}
&=
- 2 h_{012} \beta^{[2]}_{23}
- \eta^{[0]}_{012} d\beta^{[1]}_{023} 
+ \eta^{[0]}_{123} d\beta^{[1]}_{013}, \\
\gamma^{[1]}_{01234}
&=
- \eta^{[0]}_{012} \beta^{[1]}_{234}
+ h_{1234} \beta^{[1]}_{014} 
+ h_{0123} \beta^{[1]}_{034}, \\
\gamma^{[0]}_{012345}
&=
h_{2345} \beta^{[0]}_{0125}
+ h_{1234} \beta^{[0]}_{0145} 
+ h_{0123} \beta^{[0]}_{0345}, \\
c_{0123456}
&=
h_{2345} b_{01256}
+ h_{1234} b_{01456}
+ h_{0123} b_{03456},
\end{align*}
where $\check{\beta} = (b, \beta^{[0]}, \beta^{[1]},\beta^{[2]},\beta^{[3]})$ is the Deligne $4$-cocycle in Definition \ref{dfn:Deligne_4_cochain}, and we substitute $i_0 = 0, i_1 = 1, i_2 = 2, \ldots$ to suppress notations.
\end{dfn}

\begin{lem} \label{lem:cocycle_condition_gamma}
The Deligne cochain $\check{\gamma} \in C^6(\{ U_i \}, \Z(6)_D^\infty)$ satisfies:
$$
D \check{\gamma}
= 2 h \cup \check{\beta},
$$
where $h \cup \check{\beta} \in C^7(\{ U_i \}, \Z(4))$ is regarded as a cochain in $C^7(\{ U_i \}, \Z(6)_D^\infty)$.
\end{lem}

This lemma will be shown in Subsection \ref{subsec:coboundary}.

\medskip

\begin{rem}
If we define $( Q(h)_{i_0i_1i_2i_3i_4i_5} ) \in C^5(\{ U_i \}, \Z)$ by
$$
Q(h)_{i_0i_1i_2i_3i_4i_5} 
=
h_{i_2i_3i_4i_5}h_{i_0i_1i_2i_5} 
+ h_{i_1i_2i_3i_4}h_{i_0i_1i_4i_5}
+ h_{i_0i_1i_2i_3}h_{i_0i_3i_4i_5},
$$
then we have the expressions $\gamma^{[0]} = - Q(h) \cup \alpha^{[0]}$ and $c = - Q(h) \cup a$. The cochain $Q(h)$ satisfies $\delta Q(h) = - 2 h \cup h$ and represents $Sq^2(h) \in H^5(M, \Z/2)$, where $Sq^2$ is the Steenrod squaring operation.
\end{rem}

\subsection{Generalization of Mickelsson's invariant}

\begin{dfn} \label{dfn:Cech_deRham_5_cochain}
Let $g$ be a section of $P \times_{Ad} U_1(H)$, and $\{ U_i \}, s_i, \phi_{ij}, \eta^{[0]}_{ijk}, \alpha^{[0]}_i$ the choices in Definition \ref{dfn:Deligne_4_cochain}. We suppose that $\mu_3([g]) = 0$, so that:
\begin{itemize}
\item
There is $\alpha : M \to \R$ such that $\alpha^{[0]}_i = \alpha|_{U_i}$.

\item
There are $m \in Z^0(\{ U_i \}, \Z)$ and $\lambda = (\lambda^{[0]}_{ijk}, \lambda^{[1]}_{ij}, \lambda^{[2]}_i) \in C^2(\{ U_i \}, \Omega)$ such that $\beta = m \cup h + D \lambda$.
\end{itemize}
We then define a \v{C}ech-de Rham $5$-cochain
$$
\omega = (
\omega^{[0]}_{i_0i_1i_2i_3i_4i_5}, 
\omega^{[1]}_{i_0i_1i_2i_3i_4}, 
\omega^{[2]}_{i_0i_1i_2i_3}, 
\omega^{[3]}_{i_0i_1i_2}, 
\omega^{[4]}_{i_0i_1}, 
\omega^{[5]}_i
) \in C^5(\{ U_i \}, \Omega)
$$
as follows:
\begin{align*}
\omega^{[5]}_i
&= \gamma^{[5]}_i, \\
\omega^{[4]}_{i_0i_1}
&= \gamma^{[4]}_{i_0i_1}, \\
\omega^{[3]}_{i_0i_1i_2}
&= \gamma^{[3]}_{i_0i_1i_2}, \\
\omega^{[2]}_{i_0i_1i_2i_3}
&= \gamma^{[2]}_{i_0i_1i_2i_3} 
- 2 h_{i_0i_1i_2i_3} \lambda^{[2]}_{i_3}, \\
\omega^{[1]}_{i_0i_1i_2i_3i_4}
&= \gamma^{[1]}_{i_0i_1i_2i_3i_4} 
- 2 h_{i_0i_1i_2i_3} \lambda^{[1]}_{i_3i_4}, \\
\omega^{[0]}_{i_0i_1i_2i_3i_4i_5}
&= \gamma^{[0]}_{i_0i_1i_2i_3i_4i_5}
- 2 h_{i_0i_1i_2i_3} \lambda^{[0]}_{i_3i_4i_5}
- m Q(h)_{i_0i_1i_2i_3i_4i_5},
\end{align*}
where $( Q(h)_{i_0i_1i_2i_3i_4i_5} ) \in C^5(\{ U_i \}, \Z)$ is given by
$$
Q(h)_{i_0i_1i_2i_3i_4i_5} 
=
h_{i_2i_3i_4i_5}h_{i_0i_1i_2i_5} 
+ h_{i_1i_2i_3i_4}h_{i_0i_1i_4i_5}
+ h_{i_0i_1i_2i_3}h_{i_0i_3i_4i_5}.
$$
\end{dfn}

\begin{lem} \label{lem:Cech_deRham_5_cochain}
$\omega$ in Definition \ref{dfn:Cech_deRham_5_cochain} is a cocycle. 
\end{lem}

\begin{proof}
The lemma directly follows from Lemma \ref{lem:cocycle_condition_gamma} (Lemma \ref{lem:coboundary_gamma}) and the formula $\delta Q(h) = -2 h \cup h$.
\end{proof}

\begin{thm} \label{thm:mu_5}
Let $P$ be a principal $PU(H)$-bundle over a manifold $M$. Then the assignment to a section $g$ of $P \times_{Ad} U_1(H)$ of the cohomology class of $\omega$ in Lemma \ref{dfn:Cech_deRham_5_cochain} induces the following natural homomorphism
$$
\bar{\mu}^{\R}_5 : \ \mathrm{Ker}\mu_3 \longrightarrow 
H^5(M, \R)/(h(P) \cup H^2(M, \R)).
$$
\end{thm}

\begin{proof}
By definition, $\bar{\mu}_5^{\R}([g]) = [\omega]$. To prove that $\bar{\mu}_5^{\R}$ is a well-defined map, we study how $\omega$ changes according to the various choices made: For simplicity, we assume that $M$ is connected. Let $m'$ and $\lambda'$ be other choices such that $\beta = m' \cup h + D\lambda'$. We denote by $\omega$ and $\omega'$ the $5$-cocycles defined by using $\lambda, m$ and $\lambda', m'$, respectively, with the other choices unchanged. Then we have
$$
\omega' - \omega
= 
(-2 h \cup \nu^{[0]} - n Q, -2 h \cup \nu^{[1]}, -2 h \cup \nu^{[2]}, 0, 0, 0),
$$
where $n = m' - m \in Z^0(\{ U_i \}, \Z) \cong \Z$ and $\nu = (\nu^{[0]}, \nu^{[1]}, \nu^{[2]}) = \lambda' - \lambda \in C^2(\{ U_i \}, \Omega)$ obey the relation:
$$
(n h, 0, 0, 0) + D \nu = 0.
$$
In the case where $h(P) = [(h_{ijkl})] \in H^3(M, \Z)$ is not a torsion element, the relation above implies $n = 0$ and $D\nu = 0$, so that $\omega' - \omega$ is a \v{C}ech-de Rham $5$-cocycle representing an element in $h(P) \cup H^2(M, \R) \subset H^5(M, \R)$. In the case where $h(P) \in H^3(M, \Z)$ is a torsion element, there exists $(\xi_{ijk}) \in C^2(\{ U_i \}, \R)$ such that $\delta \xi = h$. Then we can express $\omega' - \omega$ as
$$
\omega' - \omega
= -2( h \cup (\nu^{[0]} + n\xi), h \cup \nu^{[1]}, h \cup \nu^{[2]}, 0, 0, 0)
+ n D(P, 0, 0, 0, 0).
$$
In the above, the $4$-cochain $(P_{i_0i_1i_2i_3i_4}) \in C^4(\{ U_i \}, \R)$ is given by
$$
P_{i_0i_1i_2i_3i_4} = \xi_{i_0i_1i_2} \xi_{i_2i_3i_4}
- h_{i_1i_2i_3i_4} \xi_{i_0i_1i_4} - h_{i_0i_1i_2i_3} \xi_{i_0i_3i_4},
$$
which satisfies $D P = 2 h \cup \xi - Q$. Because $(\nu^{[0]} + n \xi, \nu^{[1]}, \nu^{[2]}) \in C^2(\{ U_i \}, \Omega)$ is a cocycle, we conclude that $\omega' - \omega$ also represents an element in $h(P) \cup H^2(M, \R)$ in this case. If we change the choices of $\alpha$, $\eta^{[0]}$, $\phi_{ij}$ and $s_i$, then we can find suitable choices of $m$ and $\lambda$ so that $\omega$ differs by a coboundary, as is shown in Subsection \ref{subsec:change_alpha}, \ref{subsec:change_eta}, \ref{subsec:change_phi} and \ref{subsec:change_s}. Now, by the standard argument, $\bar{\mu}_5^{\R}$ is also independent of the choice of the open cover $\{ U_i \}$. As a result, $\bar{\mu}_5^{\R}$ is a natural map. Finally $\bar{\mu}_5^{\R}$ turns out to be a homomorphism by Lemma \ref{lem:additivity_beta} and \ref{lem:additivity_gamma}.
\end{proof}

Our construction only gives the homomorphism $\bar{\mu}_5^{\R}$ with values in the quotient by $h \cup H^2(M, \R)$, because $\mu_3$ is defined through a \v{C}ech-de Rham cocycle. To get a homomorphism with values in the quotient by $h \cup H^2(M, \Z)$, we may need an explicit representative of $\mu_3$ by an integral \v{C}ech cocycle.

\subsection{Characteristic class}

\begin{dfn} \label{dfn:Cech_deRham_4_cochain}
Let $P$ be any principal $PU(H)$-bundle on a manifold $M$. Given a section $g$ of $P \times_{Ad} U_1(H)$, we choose $\{ U_i \}$, $s_i$, $\phi_{ij}$, $\eta^{[0]}_{ijk}$ and $\alpha^{[0]}_i$ as in Definition \ref{dfn:Deligne_4_cochain}, and define a \v{C}ech-de Rham $4$-cochain as follows:
$$
\nu = 
(0, 0, 0, \beta^{[2]}_{ij} d\alpha^{[0]}_j, \beta^{[3]}_i d\alpha^{[0]}_i)
\in C^4(\{ U_i \}, \Omega),
$$
where $\beta^{[2]}_{ij}$ and $\beta^{[3]}_i$ are as in Definition \ref{dfn:Deligne_4_cochain}.
\end{dfn}

\begin{lem} 
$\nu$ in Definition \ref{dfn:Cech_deRham_4_cochain} is a cocycle.
\end{lem}

\begin{proof}
Notice that $d\alpha^{[0]}_i$ is the restriction of $\tr[g^{-1}dg]/(2\pi \sqrt{-1})$ to $U_i$ and appears in $\beta^{[1]}_{ijk}$ and hence in $d\beta^{[1]}_{ijk}$. Then Lemma \ref{lem:cocycle_condition_beta} (a) together with the simple fact that the square of any odd form is trivial shows $D\nu = 0$.
\end{proof}

\begin{thm} \label{thm:nu_4}
Let $P$ be any principal $PU(H)$-bundle on a manifold $M$. Then the assignment of the cohomology class of $\nu$ in Definition \ref{dfn:Cech_deRham_4_cochain} to a section $g$ of $P \times_{Ad} U_1(H)$ induces the following natural map:
$$
\nu_4 : \ K_P^1(M) \longrightarrow H^4(M, \R).
$$
\end{thm}

\begin{proof}
It suffices to verify the map is well-defined. From the results in Subsection \ref{subsec:change_alpha}, \ref{subsec:change_eta} and \ref{subsec:change_phi}, we can see that the cocycle $\nu$ is unchanged if we make other choices of $\alpha^{[0]}_i$, $\eta^{[0]}_{ijk}$ and $\phi_{ij}$. In the case where we choose other local sections $s'_i$ of $P|_{U_i}$, it holds that
\begin{align*}
{\beta'}^{[3]}_{i} - \beta^{[3]}_{i}
&= d \tau^{[2]}_i, &
{\beta'}^{[2]}_{ij} - \beta^{[3]}_{ij}
&= (\delta \tau^{[2]})_{ij}
\end{align*}
under the same notations as in Subsection \ref{subsec:change_s}. Let $\nu'$ be the cocycle defined as in Definition \ref{dfn:Cech_deRham_4_cochain} by using $\beta'$. Then we have
$$
\nu ' - \nu
= (0, 0, 0, 
\delta (\tau^{[2]} d \alpha^{[0]})_{ij}, 
d(\tau^{[2]}_i d\alpha^{[0]}_i))
= D(0, 0, 0, \tau^{[2]}_i d\alpha^{[0]}_i).
$$
Hence the cohomology class of $\nu$ is independent of the choices. It is then clear that the map $\nu_4$ is also independent of the choices of an open cover of $M$. Now the naturality is clear from the construction.
\end{proof}

\bigskip

\begin{dfn} \label{dfn:Cech_deRham_9_cochain}
Let $P$ be any principal $PU(H)$-bundle on a manifold $M$. Given a section $g$ of $P \times_{Ad} U_1(H)$, we choose $\{ U_i \}$, $s_i$, $\phi_{ij}$, $\eta^{[0]}_{ijk}$ and $\alpha^{[0]}_i$ as in Definition \ref{dfn:Deligne_4_cochain}, and define a \v{C}ech-de Rham $9$-cochain
$$
\pi = 
(0, \cdots, 0, 
\pi^{[5]}_{i_0i_1i_2i_3i_4}, 
\pi^{[6]}_{i_0i_1i_2i_3}, 
\pi^{[7]}_{i_0i_1i_2}, 
\pi^{[8]}_{i_0i_1}, 
\pi^{[9]}_i)
\in C^9(\{ U_i \}, \Omega),
$$
as follows:
\begin{align*}
\pi^{[9]}_i
&= \gamma^{[5]}_i \beta^{[3]}_i d\alpha^{[0]}_i, \\
\pi^{[8]}_{i_0i_1}
&= \gamma^{[4]}_{i_0i_1} \beta^{[3]}_{i_1} d\alpha^{[0]}_{i_1}
- \gamma^{[5]}_{i_0} \beta^{[2]}_{i_0i_1} d\alpha^{[0]}_{i_1}, \\
\pi^{[7]}_{i_0i_1i_2}
&= \gamma^{[3]}_{i_0i_1i_2} \beta^{[3]}_{i_2} d\alpha^{[0]}_{i_2}
+ \gamma^{[4]}_{i_0i_1} \beta^{[2]}_{i_1i_2} d\alpha^{[0]}_{i_2}, \\
\pi^{[6]}_{i_0i_1i_2i_3}
&= \gamma^{[2]}_{i_0i_1i_2i_3} \beta^{[3]}_{i_3} d\alpha^{[0]}_{i_3}
- \gamma^{[3]}_{i_0i_1i_2} \beta^{[2]}_{i_2i_3} d\alpha^{[0]}_{i_3}, \\
\pi^{[5]}_{i_0i_1i_2i_3i_4}
&= \gamma^{[2]}_{i_0i_1i_2i_3} \beta^{[2]}_{i_3i_4} d\alpha^{[0]}_{i_4}
+ h_{i_0i_1i_2i_3} S^{[4]}_{i_3i_4} d\alpha^{[0]}_{i_4},
\end{align*}
where $\beta^{[2]}_{ij}$ and $\beta^{[3]}_i$ are as in Definition \ref{dfn:Deligne_4_cochain}, $\gamma^{[2]}_{i_0 \cdots i_3}, \ldots, \gamma^{[5]}_i$ are as in Definition .. and $S^{[4]}_{ij} \in \Omega^4(U_{ij})$ is defined by $S^{[4]}_{ij} = \beta^{[2]}_{ij} \beta^{[2]}_{ij}$. 
\end{dfn}

\begin{lem} 
$\pi$ in Definition \ref{dfn:Cech_deRham_9_cochain} is a cocycle.
\end{lem}

\begin{proof}
It is easy to see:
\begin{align*}
d S_{ij}^{[4]}
&= 2(\beta^{[2]}_{ij}\beta^{[3]}_j - \beta^{[3]}_i\beta^{[2]}_{ij}), &
(\delta S^{[4]})_{ijk}
&= -2 \beta^{[2]}_{ij} \beta^{[2]}_{jk}
+ 2\beta^{[2]}_{ik} d\beta^{[1]}_{ijk}.
\end{align*}
These formulae and lemmas in Subsection \ref{subsec:coboundary} show the present lemma.
\end{proof}

\begin{thm} \label{thm:nu_9}
Let $P$ be any principal $PU(H)$-bundle on a manifold $M$. Then the assignment of the cohomology class of $\pi$ in Definition \ref{dfn:Cech_deRham_9_cochain} to a section $g$ of $P \times_{Ad} U_1(H)$ induces the following natural map:
$$
\nu_9 : \ K_P^1(M) \longrightarrow H^9(M, \R).
$$
\end{thm}

\begin{proof}
As in the case of $\nu_4$, it suffices to prove that the cohomology class of $\pi$ is independent of the choices of $\alpha^{[0]}_i$, $\eta^{[0]}_{ijk}$, $\phi_{ij}$ and $s_i$. The change of $\alpha^{[0]}_i$ does not alter $\pi$ by the results in Subsection \ref{subsec:change_alpha}. If we choose ${\eta'}^{[0]}_{ijk}$ instead of $\eta^{[0]}_{ijk}$, then the difference of the corresponding cocycles $\pi'$ and $\pi$ is 
$$
\pi' - \pi
= D (0, 0, 0, 0, 0, r S^{[4]} d\alpha^{[0]}, 0, 0, 0)
$$ 
under the notations in Subsection \ref{subsec:change_eta}. If we choose $\phi'_{ij}$ instead of $\phi_{ij}$, then the difference of the corresponding cocycles is
$$
\pi ' - \pi 
= - D
(0, 0, 0, 0, 0, 
\zeta^{[2]} \beta^{[2]} d\alpha^{[0]},
\zeta^{[2]} \beta^{[3]} d\alpha^{[0]} 
- \zeta^{[3]} \beta^{[2]} d\alpha^{[0]},
\zeta^{[3]} \beta^{[3]} d\alpha^{[0]},
0)
$$
under the notations in Subsection \ref{subsec:change_phi}. If we choose $s'_i$ instead of $s_i$, then the difference of the corresponding cocycles is
$$
\pi' - \pi
= D(
0, 0, 0, 0, 0, 
\upsilon^{[5]},
\upsilon^{[6]},
\upsilon^{[7]}, 
\upsilon^{[8]}
),
$$
where, under the notations in Subsection \ref{subsec:change_s}, $\upsilon^{[5]}, \ldots, \upsilon^{[8]}$ are defined by
\begin{align*}
\upsilon^{[8]}
&=
- \gamma^{[5]} \tau^{[2]} d\alpha^{[0]}
+ \xi^{[4]} {\beta'}^{[3]} d\alpha^{[0]}, \\
\upsilon^{[7]}
&=
- \gamma^{[4]} \tau^{[2]} d\alpha^{[0]}
+ \xi^{[4]} {\beta'}^{[2]} d\alpha^{[0]}
+ \xi^{[3]} {\beta'}^{[3]} d\alpha^{[0]}, \\
\upsilon^{[6]}
&= 
- \gamma^{[3]} \tau^{[2]}d\alpha^{[0]}
+ \xi^{[2]} {\beta'}^{[3]} d\alpha^{[0]}
- \xi^{[3]} {\beta'}^{[2]} d\alpha^{[0]}, \\
\upsilon^{[5]}
&=
- \gamma^{[2]} \tau^{[2]} d\alpha^{[0]}
+ \xi^{[2]} {\beta'}^{[2]} d\alpha^{[0]}
- h T^{[4]} d\alpha^{[0]},
\end{align*}
and $T^{[4]}_{ij} \in \Omega^4(U_{ij})$ is defined by $T^{[4]}_{ij} = \tau^{[2]}_{ij} \tau^{[2]}_{ij}$.
\end{proof}

\medskip

The characteristic classes $\nu_4$ and $\nu_9$ are non-trivial. This can be seen by constructing untwisted $K$-classes on products of some spheres, such as $S^1 \times S^3$. The homomorphism $\mu_1$ and the cohomology class $h(P)$ classifying $P$ induce non-trivial characteristic classes taking values in $H^1(M, \R)$ and $H^3(M, \R)$, respectively. According to Lemma \ref{lem:char_class_real}, these non-trivial characteristic classes are essentially unique ones with values in $H^p(M, \R)$, ($p = 1, 3, 5, 9$). From the lemma, there exists a non-trivial characteristic class with values in $H^8(M, \R)$. The proof of the lemma suggests that the characteristic class would be represented by a \v{C}ech-de Rham $8$-cocycle whose $8$-form part is 
$$
\frac{1}{20(2\pi\sqrt{-1})^5}
\left\{
\frac{1}{7} \tr[(g_i^{-1}dg_i)^7] \tr[g_i^{-1}dg_i]
- \frac{1}{3} \tr[(g_i^{-1}dg_i)^5] \tr[(g_i^{-1}dg_i)^3]
\right\}.
$$


\section{Comparison with AHSS}
\label{sec:comarison_with_AHSS}

\subsection{The factorization of $\mu_3$ and $\bar{\mu}_5^{\R}$}

As is shown in Lemma \ref{lem:identify_mu_1}, the homomorphism $\mu_1$ agrees with one appearing in the Atiyah-Hirzebruch spectral sequence. Consequently, the domain of $\mu_3$ is identified with $\mathrm{Ker}\mu_1 = F^3K^1_P(M)$. This subgroup fits into the exact sequence:
$$
0 \to 
F^{5}K^1_P(M) \to
F^{3}K^1_P(M) \to
E^{3, 0}_\infty \to
0.
$$
We can easily see $E^{3, 0}_\infty \subset E^{3, 0}_5 \subset  H^3(M, \Z)/(h(P) \cup H^0(M, \Z)$. Thus, we have natural homomorphisms
$$
\mathrm{Ker}\mu_1 \to
E^{3, 0}_\infty \to 
\frac{H^3(M, \Z)}{h(P) \cup H^0(M, \Z)} \to
\frac{H^3(M, \Z)}{\mathrm{Tor} + h(P) \cup H^0(M, \Z)},
$$
where the third homomorphism is induced from the quotient by the torsion subgroup $\mathrm{Tor}$ in $H^3(M, \Z)$. Tentatively, we denote by $\pi_3$ the composition of the above homomorphisms.

\begin{prop} \label{prop:AHSS_and_mu_3}
For any principal $PU(H)$-bundle over a manifold $M$, we have $\mu_3 = \pi_3$.
\end{prop}

An immediate corollary to this proposition is $F^3K^1_P(M) \subset \mathrm{Ker}\mu_3$.

\begin{proof}
It is enough to prove that $\mu_3^{\R}$ in Lemma \ref{lem:reformulate_Mickelsson_invariant} agrees with the homomorphism $\pi_3^{\R}$ given by the composition of $\pi_3$ and the homomorphism
$$
H^3(M, \Z)/(h(P) \cup H^0(M, \Z)) \to H^3(M, \R)/(h(P) \cup H^0(M, \Z))
$$
induced from the inclusion $\Z \to \R$. Let us introduce a structure of a CW complex to $M$. We write $M_{< d}$ for the union of cells of dimension less than $d$. Suppose a section $g$ of $P \times_{Ad} U_1(H)$ is such that $[g] \in \mathrm{Ker}\mu_1 = F^1K^1_P(M)$. Therefore $g|_{M_{< 3}} = 1$. On the one hand, by the construction of the Atiyah-Hirzebruch spectral sequence, $\pi_3^{\R}([g])$ is represented by a homomorphism $C_3(M) \to \R$, where $C_*(M)$ stands for the cellular chain complex. The homomorphism representing $\pi_3^{\R}([g])$ associates to a $3$-cell $e_3$ the real number given by the following isomorphisms:
$$
[g|_{e_3}] \in
K^1_{P|_{e^3}}(e_3, \partial e_3) \cong
K^1(e_3, \partial e_3) \cong
K^1(D^3, \partial D^3) \cong
K^0(\mathrm{pt}) = \Z \subset \R,
$$
where the first isomorphism is the canonical one induced from a choice of a trivialization of $P|_{e^3}$, the second isomorphism is induced from the identification of $e^3$ with the $3$-dimensional disk $D^3$, and the third isomorphism is the Thom isomorphism for the vector bundle $\R^3 \to \mathrm{pt}$. On the other hand, $\mu_3([g])$ is also represented by a homomorphism $C_3(M) \to \R$ in terms of the cellular cohomology. An application of the Atiyah-Hirzebruch spectral sequence for \v{C}ech-de Rham cohomology shows that the homomorphism $C_3(M) \to \R$ associates to $e_3$ the real number given by the following isomorphisms:
$$
[\beta|_{e^3}] \in
H^3(e_3, \partial e_3) \cong
H^3(D^3, \partial D^3) \cong
\R,
$$
where $H^*(X, Y)$ means the relative version of the \v{C}ech-de Rham cohomology, and $\beta$ is the \v{C}ech-de Rham cocycle constructed as in Lemma \ref{lem:Cech_deRham_3_cocycle}. The first isomorphism above is induced from $e_3 \cong D^3$, and the second isomorphism is the integration over $D^3$. Notice that we can represent a \v{C}ech-de Rham cocycle in $H^3(D^3, \partial D^3)$ by a closed $3$-form on $D^3$ vanishing on $\partial D^3$. Recalling the proof of Proposition \ref{prop:exmaple_three_dim}, we represent a generator of $K^1(D^3, \partial S^3)$ by a map $g_D : D^3 \to U_1(H)$ such that $g_D|_{\partial D^3} = 1$ and
$$
\int_{D^3} \frac{-1}{24\pi^2} \tr[(g_D^{-1}dg_D)^3] = 1.
$$
Since $[g|_{e_3}] = n [g_D]$ for some $n \in \Z$, we have $\pi_3([g]) = \mu_3([g])$. 
\end{proof}

\begin{cor}
For any principal $PU(H)$-bundle $P$ over a compact oriented $3$-manifold $M$, $\mu_3$ is an isomorphism. 
\end{cor}

Thus, in particular, $\mu_1$ and $\mu_3$ realize the additive isomorphism $K^1_P(M) \cong H^1(M, \Z) \oplus \Z/h$ from the Atiyah-Hirzebruch spectral sequence. Note that the corollary above can also be derived from Proposition \ref{prop:exmaple_three_dim}.

\medskip

In Subsection \ref{subsec:example}, $\mu_3$ is computed on $SU(3)$. We combine the result with Proposition \ref{prop:AHSS_and_mu_3} to reprove the result known for example in \cite{Bra,D,MMS}:

\begin{cor}
If $P \to SU(3)$ is a principal $PU(H)$-bundle such that $h = h(P) \in H^3(SU(3), \Z) \cong \Z$ is odd, then we have:
\begin{align*}
K_P^1(SU(3)) &\cong \Z/h, &
K_P^0(SU(3)) &\cong \Z/h. 
\end{align*}
\end{cor}

\begin{proof}
From the Atiyah-Hirzebruch spectral sequence, we get
$$
K_P^1(SU(3)) = \mathrm{Ker}[ d_5 : E^{3, 0}_5 \to E^{8, -4}_5 ]
\subset E^{3, 0}_5.
$$
Hence Proposition \ref{prop:AHSS_and_mu_3} implies that $\mu_3$ is injective. On the other hand, Proposition \ref{prop:example_SU(3)_odd} implies that $\mu_3$ is surjective. This concludes $K_P^1(SU(3)) \cong \Z/h$. This happens if and only if $d_5 = 0$. Now, the spectral sequence gives
$$
K_P^0(SU(3)) = \mathrm{Coker}[d_5 : E^{3, 0}_5 \to E^{8, -4}_5],
$$ 
so that $K_P^0(SU(3)) \cong E^{8, -4}_5 = H^8(SU(3), \Z)/h \cup H^5(SU(3), \Z) = \Z/h$.
\end{proof}

In the case that $h = h(P) \in H^3(SU(3), \Z) \cong \Z$ is even and non-trivial, it is known \cite{MMS} that $K_P^1(SU(3)) \cong (2\Z)/h$. But, the argument in the corollary above only proves that $K^1_P(SU(3))$ is $(2\Z)/h$ or $\Z/h$.

\medskip

From the Atiyah-Hirzebruch spectral sequence, we also have:
$$
0 \to 
F^{7}K^1_P(M) \to
F^{5}K^1_P(M) \to
E^{5, 0}_\infty \to
0.
$$

\begin{lem}
$E^{5, 0}_\infty \subset E_5^{5, 0} \subset H^5(M, \Z)/(h(P) \cup H^2(M, \Z))$. 
\end{lem}

\begin{proof}
If $h(P) \in H^3(M, \Z)$ is not a torsion element, then $E^{0, 0}_5 = 0$ so that $d_5 : E^{0, 0}_5 \to E^{5, -4}_5$ is automatically trivial. If $h(P)$ is a torsion element, then we can realize an element in $K^0_P(M)$ by a twisted vector bundle of finite rank \cite{BCMMS}, so that we have the non-trivial homomorphism $K^0_P(M) \to H^0(M, \Z)$ of taking the rank. This homomorphism is identified with the surjection $F^0K^0_P(M) \to E_\infty^{0, 0}$, so that $d_5 : E_5^{0, 0} \to E_5^{5, -4}$ is also trivial. Hence $E_5^{5, 0} \supset E_7^{5, 0} \supset \cdots \supset E_\infty^{5, 0}$.
\end{proof}

As a result, we can consider the composition $\bar{\pi}_5^{\R}$ of
$$
F^3K_P^1(M) \to
E^{5, 0}_\infty \to 
\frac{H^5(M, \Z)}{h(P) \cup H^2(M, \Z)} \to
\frac{H^5(M, \R)}{h(P) \cup H^2(M, \R)},
$$
where the third homomorphism is induced from the inclusion $\Z \to \R$. 

\begin{prop}
For any principal $PU(H)$-bundle over a manifold $M$, the restriction of $\bar{\mu}_5^{\R}$ to $F^3K_P^1(M) \subset \mathrm{Ker}\mu_3$ agrees with $2\bar{\pi}_5^{\R}$.
\end{prop}

\begin{proof}
The proof is essentially the same as that of Proposition \ref{prop:AHSS_and_mu_3}. The factor $2$ comes from the fact about the Chern character:
$$
\int_{S^5} \frac{i}{480 \pi^3} \tr[(g^{-1}dg)^5]
= 1
$$
for $g : S^5 \to U_1(H)$ representing $1 \in K^1(S^5) \cong \pi_5(U_1(H)) \cong \Z$.
\end{proof}

\subsection{Possible construction of Chern character}

As in the case of untwisted $K$-theory, there exits the notion of the \textit{Chern character} for odd twisted $K$-theory. So far, various formulations of the Chern characters are known (for example \cite{A-Se2,CMW,FHT4,M-S}). An algebro-topological method constructs the Chern character through the Atiyah-Hirzebruch spectral sequence. Hence the result about the factorizations of $\mu_3$ and $\bar{\mu}_5^{\R}$ suggests the possibility to formulate the Chern character by developing our construction with \v{C}ech-de Rham cocycles. The aim here is to justify this idea by constructing a part of a twisted \v{C}ech-de Rham cocycle from a representative of $K^1_P(M)$.

\begin{dfn} \label{dfn:chern_character_component}
Let $P$ be a principal $PU(H)$-bundle over a manifold $M$. We choose $\{ U_i \}$, $s_i$ and $\eta^{[0]}_{ijk}$ as in Subsection \ref{subsec:projective_unitary_bundle} to define $h_{ijkl} = (\delta \eta^{[0]})_{ijkl}$. We also choose $\eta^{[1]}_{ij} \in \Omega^1(U_{ij})$ and $\eta^{[2]}_i \in \Omega^2(U_i)$ so that
$$
\check{\eta} = (h_{ijkl}, \eta^{[0]}_{ijk}, \eta^{[1]}_{ij}, \eta^{[2]}_i) 
\in C^3(\{ U_i \}, \Z(3)_D^\infty)
$$
is a Deligne $3$-cocycle. For $g \in \Gamma(M, P \times_{Ad} U_1(H))$, we define \v{C}ech-de Rham cochains $\tilde{\alpha}$, $\tilde{\beta}$ and $\tilde{\gamma}$ as follows:
\begin{align*}
\tilde{\alpha} 
&= 
(0, d \alpha^{[0]}_i) 
\in C^1(\{ U_i \}, \Omega), \\
\tilde{\beta} 
&=
(0, 0, 
\beta^{[2]}_{ij} + \eta^{[1]}_{ij} d\alpha^{[0]}_j,
\beta^{[3]}_i + \eta^{[2]}_i d\alpha^{[0]}_i)
\in C^3(\{ U_i \}, \Omega), \\
\tilde{\gamma} 
&=
(0, 0, 
\gamma^{[2]} + \theta^{[2]},
\gamma^{[3]} + \theta^{[3]},
\gamma^{[4]} + \theta^{[4]},
\gamma^{[5]} + \theta^{[5]})
\in C^5(\{ U_i \}, \Omega).
\end{align*}
In the above, $\alpha^{[0]}$, $\beta^{[ \cdot ]}$ and $\gamma^{[ \cdot ]}$ are as in Subsection \ref{subsec:mu_1}, Definition \ref{dfn:Deligne_4_cochain} and \ref{dfn:Deligne_6_cochain}. The differential forms $\theta^{[2]}_{ijkl}, \theta^{[3]}_{ijk}$, $\theta^{[4]}_{ij}$ and $\theta^{[5]}_i$ are defined by:
\begin{align*}
\theta^{[5]}_0
&= 
2 \eta^{[2]}_0\beta^{[3]}_0 + \eta^{[2]}_0\eta^{[2]}_0d\alpha^{[0]}_0, \\
\theta^{[4]}_{01}
&=
2 (\eta^{[1]}_{01}\beta^{[3]}_1 + \eta^{[2]}_0\beta^{[2]}_{01})
+ (\eta^{[1]}_{01}\eta^{[2]}_1 + \eta^{[2]}_0\eta^{[1]}_{01})d\alpha^{[0]}_1,\\
\theta^{[3]}_{012}
&=
2 (\eta^{[0]}_{012}\beta^{[3]}_2 - \eta^{[1]}_{01}\beta^{[2]}_{12})
+ (
\eta^{[0]}_{012}\eta^{[2]}_2
- \eta^{[1]}_{01}\eta^{[1]}_{12}
- \eta^{[2]}_0\eta^{[2]}_{012}
) d\alpha^{[0]}_2, \\
\theta^{[2]}_{0123}
&=
2 \eta^{[0]}_{012}\beta^{[2]}_{23}
+ (
\eta^{[0]}_{012}\eta^{[1]}_{23}
- \eta^{[1]}_{01}\eta^{[0]}_{123}
+ h_{0123}\eta^{[1]}_{03}
) d\alpha^{[0]}_3,
\end{align*}
where we substitute $i = 0, j = 1, \ldots$ to suppress notations.
\end{dfn}

We write $\tilde{\eta} = (0, 0, 0, d\eta^{[2]}_i) \in Z^3(\{ U_i \}, \Omega)$ to denote the \v{C}ech-de Rham $3$-cocycle corresponding to the $3$-form $\eta$.

\begin{prop}
The cochains in Definition \ref{dfn:chern_character_component} satisfy:
\begin{align*}
D \tilde{\alpha} &= 0, &
D \tilde{\beta} &= \tilde{\eta} \wedge \tilde{\alpha}, &
D \tilde{\gamma} &= 2\tilde{\eta} \wedge \tilde{\beta}.
\end{align*}
\end{prop}

\begin{proof}
The formulae directly follow from results in Subsection \ref{subsec:coboundary}.
\end{proof}

The twisted \v{C}ech-de Rham cohomology is the cohomology of $(C^*(\{ U_i \}, D - \tilde{\eta} \wedge)$. By the proposition above, we see that the cochains $\tilde{\alpha}$, $\tilde{\beta}$ and $\tilde{\gamma}/2$ constitute a part of a twisted \v{C}ech-de Rham cocycle. 

Notice that the twisted \v{C}ech-de Rham cohomology is naturally isomorphic to the twisted de Rham cohomology $H^*_{\eta}(M)$, the cohomology of the complex $(\Omega^*(M), d - \eta)$. In \cite{CMW} the Chern character is formulated as a homomorphism
$$
\ch_{\check{\eta}} : \ 
K_P^1(M) \longrightarrow H^{\mathrm{odd}}_{\eta}(M).
$$
Thus, if we represent $\ch_{\check{\eta}}$ by odd forms $\ch_1 + \ch_3 + \ch_5 + \cdots$, then we have
\begin{align*}
d \ch_1 &= 0, &
d \ch_3 &= \eta \wedge \ch_1, &
d \ch_5 &= \eta \wedge \ch_3.
\end{align*}
Namely, the relation among $\tilde{\alpha}$, $\tilde{\beta}$ and $\tilde{\gamma}/2$ are the same as that among $\ch_1$, $\ch_3$ and $\ch_5$. At present, no explicit relation between these cochains and differential forms is known. A reason is that the model of the twisted $K$-group $K^1_P(M)$ in this paper is different from that used in the formulation of the Chern character in \cite{CMW}. However, it is plausible to expect that they are essentially the same.

\bigskip

\begin{rem}
The \v{C}ech-de Rham $3$-cochain $\tilde{\beta}$ and the Deligne $4$-cocycle $\check{\beta}$ admit a simple relation: From $\tilde{\beta} = (0, 0, \tilde{\beta}^{[2]}, \tilde{\beta}^{[3]})$, we naturally get a Deligne $4$-cocycle $(0, 0, 0, \tilde{\beta}^{[2]}, \tilde{\beta}^{[3]})$, which we write $\tilde{\beta}$ again. (We can interpret this assignment as the injection in the first short exact sequence in Proposition \ref{prop:Deligne_cohomology}). Now, we have the equality of the Deligne $4$-cocycles $\tilde{\beta} = \check{\beta} + \check{\eta} \cup \check{\alpha}$. In contrast, $\tilde{\gamma}$ and $\check{\gamma}$ seem to admit no such a simple relation.
\end{rem}


\section{Formulae for cocycle condition}
\label{sec:cocycle_condition}

This section contains some formulae related to the cocycle conditions of $\check{\beta}$ and $\check{\gamma}$ in Section \ref{sec:reformulation} and \ref{sec:generalization}. The formulae for additivity are also contained.

\subsection{Basic formulae}

\begin{lem} \label{lem:basic_formula}
Let $M$ be a manifold. The following holds true:
\begin{itemize}
\item
We have $dC_3(f) = 0$, $C_3(f^{-1}) = - C_3(f)$ and:
$$
(\delta C_3)(f, g) = 3 d B_2(f, g).
$$

\item
$B_2$ is a group cocycle: If $f$, $g$ or $h$ takes values in $U_1(H)$, then we have
$$
(\delta B_2)(f, g, h) = 0.
$$

\item
For $u : M \to U(1)$, we have
\begin{align*}
B_2(uf, g) &= B_2(f, g) + (u^{-1}du) \wedge \tr[g^{-1}dg], \\
B_2(f, ug) &= B_2(f, g) - (u^{-1}du) \wedge \tr[f^{-1}df].
\end{align*}
\end{itemize}
\end{lem}

We here remark conventions used in the lemma above: As usual, a group $p$-cochain of $C^\infty(M, U(H))$ with values in $\Omega^q(M)$ means a map
$$
K : \overbrace{C^\infty(M, U(H)) \times \cdots \times C^\infty(M, U(H))}^p 
\longrightarrow \Omega^q(M).
$$
Its coboundary is the group $(p+1)$-cochain $\delta K$ given by
\begin{align*}
(\delta K)(f_1, \ldots, f_{p+1})
&= K(f_2, \ldots, f_{p+1}) \\
& \quad + \sum_{i = 1}^p (-1)^i
K(f_1, \ldots, f_{i-1}, f_if_{i+1}, f_{i+2}, \ldots, f_{p+1}) \\
& \quad + (-1)^{p+1} f(f_1, \ldots, f_p).
\end{align*}
Then, in the second formula of the lemma, we interpret $(\delta B_2)(f, g, h)$ as
\begin{multline*}
(\delta B_2)(f, g, h)
= B_2(g, h) - B_2(fg, h) + B_2(f, gh) - B_2(f, g) \\
=
\tr[
g^{-1}dg dh h^{-1} 
- (fg)^{-1}d(fg) dh h^{-1}
+ f^{-1}df d(gh) (gh)^{-1}
- f^{-1}df dg g^{-1}].
\end{multline*}
Notice that if, for example, neither $g$ nor $h$ takes values in $U_1(H)$, then neither $dg$ nor $dh$ takes values in trace class operators. Therefore each $B_2(g, h)$ and $B_2(fg, h)$ does not make sense, because the differential forms
\begin{align*}
& g^{-1}dg dh h^{-1}, &
& (fg)^{-1}d(fg) dh h^{-1}
\end{align*}
do not generally take values in the trace class operators. However, the difference of these differential forms takes values in the trace class operators. This is the key to the interpretation of $\delta B_2$. The same type of interpretations, such as
$$
B_2(f_1, g_1) + B_2(f_2, g_2) + \cdots
= \tr[f_1^{-1}df_1 \wedge dg_1 g_1^{-1} 
+ f_2^{-1}df_2 \wedge dg_2 g_2^{-1} + \cdots
],
$$
are adapted throughout this paper.

\begin{proof}
For the first formula, we put $F = f^{-1}df$ and $\bar{F} = dff^{-1}$ for $f : M \to U_1(H)$. These $1$-forms satisfy the Maurer-Cartan equations: $dF + F^2 = 0$ and $d\bar{F} - \bar{F}^2 =0$. We also put $G = g^{-1}dg$ and $\bar{G} = dgg^{-1}$. Then we have
$$
(fg)^{-1}d(fg) = g^{-1}f^{-1}(df g + fdg) = F^g + G,
$$
where $F^g = g^{-1}Fg$. Now, we have
\begin{align*}
C_3(fg) 
&=
\tr [ (F^g + G)^3 ]
=
\tr [ F^3 ] + \tr [ G^3 ] + 3 \tr[ F^2 \bar{G} + F \bar{G}^2 ], \\
d \tr[ F \bar{G} ]
&= \tr[ dF \bar{G} - F d \bar{G} ]
= \tr[ - F^2\bar{G} - F \bar{G}^2 ]
= - \tr[ F^2\bar{G} + F \bar{G}^2 ].
\end{align*}
Thus, $C_3(fg) = C_3(f) + C_3(g) - 3 d \tr[ F \bar{G} ]$, the first formula. 

For the second formula, we recall our definition of $(\delta B_2)(f, g, h)$:
\begin{multline*}
(\delta B_2)(f, g, h)
= B_2(g, h) - B_2(fg, h) + B_2(f, gh) - B_2(f, g) \\
=
\tr[
g^{-1}dg dh h^{-1} 
- (fg)^{-1}d(fg) dh h^{-1}
+ f^{-1}df d(gh) (gh)^{-1}
- f^{-1}df dg g^{-1}].
\end{multline*}
We can express the differential form in the trace above as
$$
G\bar{H} - (F^g + G)\bar{H} + F(\bar{G} + \bar{H}^{g^{-1}}) - F\bar{G} \\
= (Fg\bar{H})g^{-1} - g^{-1}(Fg\bar{H}).
$$
If $f$ or $h$ take values in $U_1(H)$, then $Fg\bar{H}$ is of trace class, so that
$$
\tr[(Fg\bar{H})g^{-1} - g^{-1}(Fg\bar{H})]
= \tr[g^{-1}(Fg\bar{H}) - g^{-1}(Fg\bar{H})]
= 0.
$$
If $g$ takes values in $U_1(H)$, then $g^{-1} = 1 + t$ with $t$ a map with its values in trace class operators, so that:
$$
\tr[(Fg\bar{H})g^{-1} - g^{-1}(Fg\bar{H})]
= \tr[(Fg\bar{H})t - t(Fg\bar{H})]
= \tr[t(Fg\bar{H}) - t(Fg\bar{H})] 
= 0.
$$
Hence $(\delta B_2)(f, g, h) = 0$ is shown. The other formulae are straightforward.
\end{proof}

\begin{lem} \label{lem:conjugate_C3}
For $g : M \to U_1(H)$ and $\phi : M \to U(H)$, we have
$$
C_3(\phi^{-1} g \phi) - C_3(g) 
= 3 d \{ B_2(\phi^{-1}g\phi, \phi^{-1}) - B_2(\phi^{-1}, g) \}.
$$
\end{lem}

\begin{proof}
As before, we put $\Phi = \phi^{-1}d\phi$, $\bar{\Phi} = d\phi \phi^{-1}$, $G = g^{-1}dg$ and $\bar{G} = dg g^{-1}$. We also put $\psi = \phi^{-1}g\phi$ and $\Phi' = \Phi^{\psi} = \psi^{-1} \Phi \psi$. Then we have
\begin{align*}
C_3(\phi^{-1}g\phi) - C_3(g)
&= \tr [ (\Phi - \Phi' + G^\phi)^3] - \tr[G^3] \\
&= \tr [ 
(\Phi - \Phi')^3 + 3 G^\phi (\Phi - \Phi')^2 + 3 (G^2)^\phi (\Phi - \Phi')
] \\
&=
\tr[(\Phi - \Phi')^3]
+ 3 \tr [ G(\bar{\Phi} - \bar{\Phi}^g)^2]
+ 3 \tr [ G^2(\bar{\Phi} - \bar{\Phi}^g) ],
\end{align*}
noting that $\Phi - \Phi'$ is of trace class. By a careful computation, we obtain:
$$
\tr[(\Phi - \Phi')^3]
= 3 \tr[ (\Phi')^2\Phi - \Phi^2\Phi' ]
= 3 \tr[ (\bar{\Phi}^g)^2\bar{\Phi} - \bar{\Phi}^2 \bar{\Phi}^g].
$$
On the other hand, we have the expression:
$$
B_2(\phi^{-1}g\phi, \phi^{-1}) - B_2(\phi^{-1}, g)
=
\tr[ (\bar{\Phi}^g - G - \bar{\Phi})\bar{\Phi} + \bar{\Phi} \bar{G} ].
$$
Now, we have $ d \bar{\Phi}^2 = 0$ and 
\begin{align*}
d ((\bar{\Phi}^g - G)\bar{\Phi} + \bar{\Phi} \bar{G})
&= 
((\bar{\Phi}^g)^2 \bar{\Phi} - \bar{\Phi}^g \bar{\Phi}^2)
+ (G^2 \bar{\Phi} - \bar{\Phi} g G^2 g^{-1}) \\
& \quad
+
(\bar{\Phi}^2 g G g^{-1} - G \bar{\Phi}^g \bar{\Phi} 
- \tilde{\Phi}^g G \bar{\Phi} + G \bar{\Phi}^2).
\end{align*}
Taking the trace, we get
\begin{align*}
\tr[((\bar{\Phi}^g)^2 \bar{\Phi} - \bar{\Phi}^g \bar{\Phi}^2)]
&=
\tr[ (\bar{\Phi}^g)^2\bar{\Phi} - \bar{\Phi}^2 \bar{\Phi}^g], \\
\tr [ G^2 \bar{\Phi} - \bar{\Phi} g G^2 g^{-1} ]
&= \tr[ G^2(\bar{\Phi} - \bar{\Phi}^g) ], \\
\tr[
\bar{\Phi}^2 g G g^{-1} - G \bar{\Phi}^g \bar{\Phi} 
- \tilde{\Phi}^g G \bar{\Phi} + G \bar{\Phi}^2
]
&=
\tr[ G(\bar{\Phi}^g - \bar{\Phi})^2 ].
\end{align*}
Thus, the lemma is proved.
\end{proof}

\bigskip

\begin{lem} \label{lem:key_formulae}
The following holds true:
\begin{itemize}
\item
We have $dC_5(f) = 0$, $C_5(f^{-1}) = - C_5(f)$ and:
$$
(\delta C_5)(f, g) = 5 d B_4(f, g)
$$

\item
$B_4 = B$ is a group cocycle up to an exact form:
$$
(\delta B_4)(f, g, h) = dA(f, g, h).
$$

\item
$A = A_3$ is a group cocycle:
$$
(\delta A)(f, g, h, k) = 0.
$$

\item
For $u : M \to U(1)$, we have
\begin{align*}
B_4(uf, g) 
&= B_4(f, g)
+ (u^{-1}du) \wedge \{ C_3(g) - dB_2(f, g) \}, \\
B_4(f, ug)
&= B_4(f, g)
- (u^{-1}du) \wedge \{ C_3(f) - dB_2(f, g) \}.
\end{align*}

\item
If $u : M \to U(1)$, then we have
\begin{align*}
A(uf, g, h)
&= A(f, g, h)
+ 2 (u^{-1}du) \cdot B_2(g, h), \\
A(f, ug, h)
&=
A(f, g, h)
- 2 (u^{-1}du) \cdot B_2(fg, h)
+ 2 (u^{-1}du) \cdot B_2(g, h), \\
&=
A(f, g, h)
- 2 (u^{-1}du) \cdot B_2(f, gh)
+ 2 (u^{-1}du) \cdot B_2(f, g), \\
A(f, g, uh)
&=
A(f, g, h)
+ 2 (u^{-1}du) \cdot B_2(f, g).
\end{align*}

\end{itemize}
\end{lem}

\begin{proof}
Under the notations $F = f^{-1}df$, $\bar{F} = dff^{-1}$, etc.\ in the proof of Lemma \ref{lem:basic_formula} and \ref{lem:conjugate_C3}, we express $C_5(f)$, $B_4(f, g)$ and $A(f, g, h)$ in Definition \ref{dfn:group_5_cochains} as
\begin{align*}
C_5(f) 
&= \tr [ F^5 ], \\
B_4(f, g)
&= \tr [ F\bar{G}^3 + \frac{1}{2}(F\bar{G})^2 + F^3\bar{G} ], \\
A(f, g, h)
&= \tr [ F^g(G\bar{H} - \bar{H}G) ].
\end{align*}
We can then compute $C_5(fg) = \tr[(F^g + G)^5]$ to have
\begin{align*}
& 
C_5(fg) - C_5(f) - C_5(g) \\
& \quad
= 5 \tr [
F\bar{G}^5 + (F^2 \bar{G}^3 + F\bar{G}F\bar{G}^2) 
+ (F^3\bar{G}^2 + F^2\bar{G}F\bar{G}) + F^4\bar{G}
] \\
& \quad
=
5 \tr [
F d\bar{G} \bar{G}^2 - dF \bar{G}^3
+ F\bar{G}Fd\bar{G}
- dFF\bar{G}^2
- dF\bar{G}F\bar{G}
- dF F^2 \bar{G}
] \\
& \quad
=
5 d \tr [ 
- F\bar{G}^3
- \frac{1}{2} F\bar{G}F\bar{G}
- F^3 \bar{G}
], 
\end{align*}
showing $\delta C_5 = - 5d B_4$. Next, we obtain by computations:
\begin{align*}
(\delta B_4)(f, g, h)
&=
\tr[
F \bar{G} \bar{H}^{g^{-1}} \bar{G}
+ F\bar{G}(\bar{H}^{g^{-1}})^2
+ F(\bar{H}^{g^{-1}})^2\bar{G}
] \\
& \quad
- 
\tr [
(F^g)^2G\bar{H} + G (F^g)^2 \bar{H} + GF^gG\bar{H}
], \\
&=
\tr [ 
- F dg \bar{H} dg^{-1}
+ F dg \bar{H}^2 g^{-1}
- F g \bar{H}^2 dg^{-1}
] \\
& \quad
+ 
\tr [
- g^{-1} F^2 dg \bar{H}
+ dg^{-1} F^2 g \bar{H}
+ dg^{-1} F dg \bar{H}
].
\end{align*}
Now, we use the Maurer-Cartan equations to see
\begin{align*}
d(F g \bar{H} dg^{-1})
&=
- F^2 g \bar{H} dg^{-1} 
- F dg \bar{H} dg^{-1}
- F g \bar{H}^2 dg^{-1}, \\
d(F dg \bar{H} g^{-1})
&=
- F^2 dg \bar{H} g^{-1}
+ F dg \bar{H}^2 g^{-1}
- F dg \bar{H} dg^{-1}.
\end{align*}
Therefore we conclude
$$
(\delta B_4) (f, g, h)
=
d \tr [ F g \bar{H} dg^{-1} + F dg \bar{H} g^{-1} ],
$$
showing $\delta B_4 = d A$. The remaining formulae are easier to prove.
\end{proof}

\begin{lem} \label{lem:conjugate_C5}
For $g : M \to U_1(H)$ and $\phi : M \to U(H)$, we have
$$
C_5(\phi^{-1} g \phi) - C_5(g) 
= 5 d \{ B_4(\phi^{-1}g\phi, \phi^{-1}) - B_4(\phi^{-1}, g) \}.
$$
\end{lem}

This lemma can be shown in the same way of calculations as in Lemma \ref{lem:conjugate_C3}.

\subsection{The coboundary of $\check{\beta}$ and $\check{\gamma}$}
\label{subsec:coboundary}

We here prove Lemma \ref{lem:cocycle_condition_beta} (a):

\begin{lem} 
$\check{\beta}$ in Definition \ref{dfn:Deligne_4_cochain} is a cocycle:
\begin{align*}
(\delta \beta^{[3]})_{ij} - d\beta^{[2]}_{ij} &= 0, \\
(\delta \beta^{[2]})_{ijk} + d\beta^{[1]}_{ijk} &= 0, \\
(\delta \beta^{[1]})_{ijkl} - d\beta^{[0]}_{ijkl} &= 0, \\
(\delta \beta^{[0]})_{ijklm} + b_{ijklm} &= 0, \\
(\delta b)_{ijklmn} &= 0.
\end{align*}
\end{lem}

\begin{proof}
We have $(\delta \beta^{[3]})_{ij} = d\beta^{[2]}_{ij}$ directly from Lemma \ref{lem:conjugate_C3}. We put $\tilde{\beta}^{[2]}_{ij} = -8\pi^2 \beta^{[2]}_{ij} = B(g_j, \phi_{ji}) - B(\phi_{ji}, g_i)$. Recall that $\phi_{ij} \phi_{jk} = f_{ijk} \phi_{ik}$ with $f_{ijk}$ taking values in $U(1)$. We use Lemma \ref{lem:basic_formula} to have
\begin{align*}
\delta (\tilde{\beta}^{[2]})_{ijk}
&=
2 (f^{-1}_{ijk}df_{ijk}) \wedge 
\{ \tr[g_i^{-1}dg_i] + \tr[ g_k^{-1}dg_k ] \} \\
&
- \delta B(g_k, \phi_{kj}, \phi_{ji})
+ \delta B(\phi_{kj}, g_j, \phi_{ji})
- \delta B(\phi_{kj}, \phi_{ji}, g_i).
\end{align*}
Because $\tr[ g_i^{-1}dg_i] = \tr[ g_k^{-1}dg_k]$ on $U_{ijk}$, we conclude $(\delta \beta^{[2]})_{ijk} = - d\beta^{[1]}_{ijk}$. The remaining cocycle conditions are much easier to check.
\end{proof}

Next, we prove Lemma \ref{lem:cocycle_condition_gamma}:

\begin{lem} \label{lem:coboundary_gamma}
$\check{\gamma}$ in Definition \ref{dfn:Deligne_6_cochain} satisfies $D \check{\gamma} = 2 h \cup \check{\beta}$, namely,
\begin{align*}
\delta \gamma^{[5]} - d\gamma^{[4]} &= 0, \\
\delta \gamma^{[4]} + d\gamma^{[3]} &= 0, \\
\delta \gamma^{[3]} - d\gamma^{[2]} &= -2 h \cup \beta^{[3]}, \\
\delta \gamma^{[2]} + d\gamma^{[1]} &= -2 h \cup \beta^{[2]}, \\
\delta \gamma^{[1]} - d\gamma^{[0]} &= -2 h \cup \beta^{[1]}, \\
\delta \gamma^{[0]} + c &= -2 h \cup \beta^{[0]}, \\
\delta c &= -2 h \cup b.
\end{align*}
\end{lem}

\begin{proof}
The direct computations basically proves this lemma, so we only indicates the points of computations: $\delta \gamma^{[5]} = d \gamma^{[4]}$ follows from Lemma \ref{lem:conjugate_C5}, and $\delta \gamma^{[4]} = d \gamma^{[3]}$ from computations by using Lemma \ref{lem:key_formulae}. To show $\delta \gamma^{[3]} = -2 h \cup \beta^{[3]} + d \gamma^{[4]}$, it may be better to introduce some notations for simplicity: We will write $C = C_3$, $B = B_2$ and $\eta = \eta^{[0]}$. We put $\tilde{\gamma}^{[3]}_{012} = - 48\pi^3\sqrt{-1}\gamma^{[3]}_{012}$, where $i_0 = 0$, $i_1 = 1, \ldots$ to suppress notations. We then express $\tilde{\gamma}^{[3]}_{012}$ as follows:
$$
\tilde{\beta}^{[3]}_{0123}
= 
-2(2\pi\sqrt{-1}) \eta^{[0]}_{012} C(g_2)
- 3(2\pi \sqrt{-1}) P_{012} 
- (2\pi\sqrt{-1}) Q_{012}
+ R_{012},
$$
where $P_{012}$, $Q_{012}$ and $R_{012}$ are
\begin{align*}
P_{012}
&= d \eta_{012} \tilde{\beta}^{[2]}_{02}
= d  \eta_{012} \{ B(g_2, \phi_{20}) - B(\phi_{20}, g_0) \}, \\
Q_{012}
&= B(\phi_{21}\phi_{10}, g_0) + B(g_2, \phi_{21}\phi_{10}), \\
R_{012}
&= A(g_2, \phi_{21}, \phi_{10}) 
- A(\phi_{21}, g_1, \phi_{10}) + A(\phi_{21}, \phi_{10}, g_0).
\end{align*}
Now, we can prove
\begin{align*}
(\delta P)_{012}
&= d \eta_{123} 
\{ - \tilde{\beta}^{[2]}_{01} + (\delta \tilde{\beta}^{[2]})_{013} \}
+ d \eta_{012} 
\{ - \tilde{\beta}^{[2]}_{23} + (\delta \tilde{\beta}^{[2]})_{023} \}, \\
(\delta Q)_{0123}
& =
d \eta_{123} 
\{ 
B(\phi_{32}\phi_{21}, g_1) + B(g_3, \phi_{32}\phi_{21}) 
\} \\
& \
+
d \eta_{123} 
\{
- B(\phi_{32}\phi_{21}\phi_{10}, g_0) - B(g_3, \phi_{32}\phi_{21}\phi_{10})
\} \\
& \
+ d \eta_{012} 
\{ 
B(\phi_{32}\phi_{21}\phi_{10}, g_0) 
+ B(g_3, \phi_{32}\phi_{21}\phi_{10}) 
\} \\
& \
+
d \eta_{012} 
\{
- B(\phi_{21}\phi_{10}, g_0) 
- B(g_2, \phi_{21}\phi_{10})
\}, \\
(\delta R)_{0123}
&=
2(2\pi i) d\eta_{012}
\{ 
- \tilde{\beta}^{[2]}_{23}
+ B(\phi_{32}\phi_{21}\phi_{10}, g_0)
- B(\phi_{21}\phi_{10}, g_0)
\} \\
& \
+
2(2\pi i) d\eta_{123}
\{ 
- \tilde{\beta}^{[2]}_{01}
- B(g_3, \phi_{32}\phi_{21}\phi_{10})
+ B(g_3, \phi_{32}\phi_{21})
\} \\
& \
+ (\delta A)(g_3, \phi_{32}, \phi_{21}, \phi_{10})
- (\delta A)(\phi_{32}, g_2, \phi_{21}, \phi_{10}) \\
& \
+ (\delta A)(\phi_{32}, \phi_{21}, g_1, \phi_{10})
- (\delta A)(\phi_{32}, \phi_{21}, \phi_{10}, g_0).
\end{align*}
In particular, we have
$$
\delta( - (2\pi\sqrt{-1}) Q + R)_{0123}
= 
(2\pi\sqrt{-1})d\eta_{012}\{ -3\tilde{\beta}^{[2]}_{23} \}
+(2\pi\sqrt{-1})d\eta_{012}\{ -3\tilde{\beta}^{[2]}_{01} \},
$$
which eventually leads to the formula showing $\delta \gamma^{[3]} = -2 h \cup \beta^{[3]} + d \gamma^{[4]}$:
\begin{align*}
(\delta \tilde{\gamma}^{[3]})_{0123}
&=
(2\pi\sqrt{-1})
\{
- h_{0123} C(g_3)
\} \\
& \quad
+
(2\pi\sqrt{-1})
\{
- 6 d(\eta_{012}\tilde{\beta}^{[2]}_{23})
+ 3 (
d\eta_{012} (\delta \tilde{\beta}^{[2]})_{023}
- d\eta_{023} (\delta \tilde{\beta}^{[2]})_{013}
)
\}.
\end{align*}
In order to prove $\delta \gamma^{[2]} = -2 h \cup \beta^{[2]} + d\gamma^{[1]}$ and $\delta \gamma^{[1]} = -2 h \cup \beta^{[1]} - d\gamma^{[0]}$, the following formulae are helpful: Let $k_{0123}$ and $\ell_{01234}$ be
\begin{align*}
k_{0123} &= \eta_{012}d\eta_{023} - \eta_{123}d\eta_{013}, &
\ell_{01234} &= h_{1234} \eta_{014} + h_{0123} \eta_{034}.
\end{align*}
Then it holds that:
\begin{align*}
(\delta k)_{01234}
&= \eta_{012}d\eta_{234} - \eta_{234}d\eta_{012} \\
& \quad
- \eta_{234} \delta(d\eta)_{0124}
+ \eta_{123} \delta(d\eta)_{0134}
- \eta_{012} \delta(d\eta)_{0234} \\
& \quad
+ (\delta \eta)_{1234} d\eta_{014} + (\delta \eta)_{0123} d\eta_{034}, \\
(\delta \ell)_{012345}
&= h_{2345} \eta_{012} - h_{0123} \eta_{345} \\
& \quad
- (\delta h)_{12345} \eta_{015} + (\delta h)_{01234} \eta_{045} \\
& \quad
+ h_{2345} (\delta \eta)_{0125}
+ h_{1234} (\delta \eta)_{0145}
+ h_{0123} (\delta \eta)_{0345}.
\end{align*}
Notice that we have $\gamma^{[0]} = - Q(h) \cup \alpha^{[0]}$ and $c = - Q(h) \cup a$, where 
$$
Q(h)_{012345}
= h_{2345}h_{0125} + h_{1234}h_{0145} + h_{0123}h_{0345}
$$
as in Definition \ref{dfn:Cech_deRham_5_cochain}. By using $\delta Q(h) = -2 h \cup h$, we can readily verify the remaining formulae $\delta \gamma^{[0]} = -2 h \cup \beta^{[0]} - c$ and $\delta c = -2 h \cup b$.
\end{proof}

\subsection{Additivity}

We here verify how the assignments of the cochains $\check{\beta}$ and $\check{\gamma}$ to a section $g \in \Gamma(M, P \times_{Ad} U_1(H))$ behave with respect to the multiplication in the group $\Gamma(M, P \times_{Ad} U_1(H))$. For this aim, we choose and fix a good cover $\{ U_i \}$ of $M$, local trivializations $s_i$ of $P$, lifts $\phi_{ij}$ of transition functions $\bar{\phi}_{ij}$, and lifts $\eta^{[0]}_{ijk}$ of $f_{ijk}$. Suppose that $g$ and $g'$ are sections of $P \times_{Ad} U_1(H)$. Accordingly, we have the local expressions $\{ g_i \}$ and $\{ g'_i \}$ of $g$ and $g'$, respectively. We then choose lifts $\alpha^{[0]}_i$ and ${\alpha'}^{[0]}_i$ of $\mathrm{det}[g_i]$ and $\mathrm{det}[g'_i]$, respectively. From these data, we get the Deligne cocycles $\check{\beta}$ and $\check{\beta}'$ as in Definition \ref{dfn:Deligne_4_cochain}. Now, the product $g'' = g'g$ is also a section, whose local expression $\{ g''_i \}$ is given by $g''_i = g'_i g_i$. We choose a lift ${\alpha''}^{[0]}_i$ of $\mathrm{det}[g''_i]$ to be ${\alpha''}^{[0]}_i = {\alpha'}^{[0]}_i + \alpha^{[0]}_i$, and let $\check{\beta}''$ be the corresponding Deligne cocycle.

\begin{lem} \label{lem:additivity_beta}
We have
$$
\check{\beta}'' - \check{\beta}' - \check{\beta}
= (0, 0, 0, (\delta \sigma^{[2]})_{ij}, d \sigma^{[2]}_i)
= D(0, 0, 0, \sigma^{[2]}_i),
$$
where $\sigma^{[2]}_i \in \Omega^2(U_i)$ is defined by
$$
\sigma^{[2]}_i
= \frac{1}{8\pi^2} B_2(g'_i, g_i).
$$
\end{lem}

\begin{proof}
It suffices to prove the formulae:
\begin{align*}
{\beta''}^{[3]}_i - {\beta'}^{[3]}_i - \beta^{[3]}_i
&= d \tau^{[2]}_i, &
{\beta''}^{[2]}_{ij} - {\beta'}^{[2]}_{ij} - \beta^{[2]}_{ij}
&= (\delta \sigma^{[2]})_{ij}.
\end{align*}
The first formula immediately follows from the cocycle condition for $C_3$. For the second formula, we get the following formula by computations:
\begin{multline*}
8\pi^2 \{ 
{\beta''}^{[2]}_{ij} - {\beta'}^{[2]}_{ij} - \beta^{[2]}_{ij} 
- (\delta \sigma^{[2]}_{ij})
\} \\
= 
\delta B_2(\phi_{ji}, g'_i, g_i)
- \delta B_2(g'_j, \phi_{ji}, g_i) 
+ \delta B_2(g'_j, g_j, \phi_{ji}).
\end{multline*}
Now $\delta B_2 = 0$ completes the proof.
\end{proof}

We let $\check{\gamma}, \check{\gamma}'$ and $\check{\gamma}''$ be the Deligne cochain defined as in Definitions \ref{dfn:Deligne_6_cochain} corresponding to the choices $\alpha^{[0]}_i$, ${\alpha'}^{[0]}_i$ and ${\alpha''}^{[0]}_i$, respectively.

\begin{lem} \label{lem:additivity_gamma}
We have
$$
\check{\gamma}'' - \check{\gamma}' - \check{\gamma}
= (0, 0, 0, 2 h \cup \sigma^{[2]}, 0, 0, 0)
+ D(0, 0, 0, \xi^{[2]}, \xi^{[3]}, \xi^{[4]}),
$$
where $\xi^{[4]}_{i_0} \in \Omega^4(U_{i_0}), \xi^{[3]}_{i_0i_1} \in \Omega^3(U_{i_0i_1})$ and $\xi^{[2]}_{i_0i_1i_2} \in \Omega^2(U_{i_0i_1i_2})$ are defined by
\begin{align*}
\xi^{[4]}_0
&= 
\frac{-i}{48\pi^3} B_4(g'_0, g_0), \\
\xi^{[3]}_{01}
&= 
\frac{i}{48\pi^3}
\{
A(\phi_{10}, g'_0, g_0) - A(g'_1, \phi_{10}, g_0) + A(g'_1, g_1, \phi_{10})
\}, \\
\xi^{[2]}_{012}
&= 
-2 \eta_{012} \sigma^{[2]}_2.
\end{align*}
In the above, $i_0 = 0, i_1 = 1, \ldots$ to suppress notations.
\end{lem}

\begin{proof}
The claim in the lemma is equivalent to the following formulae:
\begin{align*}
{\gamma''}^{[5]} - {\gamma'}^{[5]} - \gamma^{[5]}
&=
d \xi^{[4]}, \\
{\gamma''}^{[4]} - {\gamma'}^{[4]} - \gamma^{[4]}
&=
\delta \xi^{[4]} - d \xi^{[3]}, \\
{\gamma''}^{[3]} - {\gamma'}^{[3]} - \gamma^{[3]}
&=
\delta \xi^{[3]} + d \xi^{[2]}, \\
{\gamma''}^{[2]} - {\gamma'}^{[2]} - \gamma^{[2]}
&=
\delta \xi^{[2]} + 2 h \cup \sigma^{[2]} , \\
{\gamma''}^{[1]} - {\gamma'}^{[1]} - \gamma^{[1]} &= 0, \\
{\gamma''}^{[0]} - {\gamma'}^{[0]} - \gamma^{[0]} &= 0, \\
c'' - c' - c &= 0.
\end{align*}
We get the first formula from Lemma \ref{lem:conjugate_C5}. Next, we have by computations:
\begin{multline*}
{\gamma''}^{[4]}_{01} - {\gamma'}^{[4]}_{01} - \gamma^{[4]}_{01} 
- (\delta \xi^{[4]})_{01} \\
= 
\frac{-i}{48\pi^3}
\{
\delta B_4(\phi_{10}, g'_0, g_0)
- \delta B_4(g'_1, \phi_{10}, g_0) 
+ \delta B_4(g'_1, g_1, \phi_{10})
\}.
\end{multline*}
This leads to the second formula, since $\delta B_4 = d A$ as in Lemma \ref{lem:key_formulae}. We then use $\delta A = 0$ to get the identity:
\begin{align*}
0
&
= \delta A(\phi_{21}, \phi_{10}, g'_0, g_0)
- \delta A(\phi_{21}, g'_1, \phi_{10}, g'_0)
+ \delta A(\phi_{21}, g'_1, g_1, \phi_{10}) \\
& 
+ \delta A(g'_2, \phi_{21}, \phi_{10}, g_0) \}
- \delta A(g'_2, \phi_{21}, g_1, \phi_{10})
+ \delta A(g'_2, g_2, \phi_{21}, \phi_{10}).
\end{align*}
Expanding the identity above, we obtain
\begin{multline*}
{\gamma''}^{[3]}_{012} - {\gamma'}^{[3]}_{012} - \gamma^{[3]}_{012} 
- (\delta \xi^{[3]})_{012} \\
= 
-2 \eta_{012} d \sigma^{[2]}_2
- d \eta^{[0]}_{012} \{ \sigma^{[2]}_2 - \sigma^{[2]}_0 \}
+ d \eta^{[0]}_{012} \{ - \sigma^{[2]}_0 - \sigma^{[2]}_2 \},
\end{multline*}
which leads to the third formula. The fourth formula is verified easily by using Lemma \ref{lem:additivity_beta}. The remaining formulae are clear by construction. 
\end{proof}


\section{Formulae for well-definedness}
\label{sec:well_definedness}

This section contains formulae for the proof of the well-definedness of $\mu_3$, $\bar{\mu}_5^{\R}$, $\nu_4$ and $\nu_9$. In particular, we give formulae describing how the Deligne cochains in Definition \ref{dfn:Deligne_4_cochain} and Definition \ref{dfn:Deligne_6_cochain} change according to the choices made.

\subsection{The change of $\alpha^{[0]}_i$}
\label{subsec:change_alpha}

Among the choices in Definition \ref{dfn:Deligne_4_cochain}, we here consider to change the choice of $\alpha^{[0]}_i$. Let ${\alpha'}^{[0]}_i$ be another choice such that $\exp 2\pi\sqrt{-1} {\alpha'}^{[0]}_i = \det[g_i]$. The difference $s_i = {\alpha'}^{[0]}_i - \alpha^{[0]}_i$ defines a cochain $(s_i) \in C^0(\{ U_i \}, \Z(1)_D^\infty)$, which satisfies
$$
\check{\alpha}' - \check{\alpha}
= (a'_{ij}, {\alpha'}^{[0]}_i) - (a_{ij}, \alpha^{[0]}_i)
= ((\delta s)_{ij}, s_i) = D (s_i, 0).
$$

\begin{lem} \label{lem:change_alpha_in_beta}
Let $\check{\beta'}$ be the $4$-cocycle defined as in Definition \ref{dfn:Deligne_4_cochain} under the choice of ${\alpha'}^{[0]}_i$ above with the other choices unchanged. Then we have
\begin{align*}
\check{\beta}' - \check{\beta}
&= (b' - b, 
{\beta'}^{[0]} - \beta^{[0]}, 0, 0, 0) \\
&= ( \delta (h \cup s), - h \cup s, 0, 0, 0)
= D( h \cup s, 0, 0, 0).
\end{align*}
Let $\check{\gamma}'$ be the $6$-cochain defined as in Definition \ref{dfn:Deligne_6_cochain} similarly. Then we have
$$
\check{\gamma}' - \check{\gamma} 
- 2 (h \cup s, 0, 0, 0, 0, 0, 0)
= D (t, 0, 0, 0, 0, 0),
$$
where $t_{i_0 \cdots i_5} \in \Z$ is given by $t_{i_0 \cdots i_5} = - c'_{i_0 \cdots i_5} + c_{i_0 \cdots i_5} = (Q(h) \cup s)_{i_0 \cdots i_5}$.
\end{lem}

\begin{proof}
The formula for $\check{\beta}' - \check{\beta}$ is easy. This formula and the expression
$$
{c'}^{[0]}_{i_0 \cdots i_6} - c_{i_0 \cdots i_6} 
- 2(h \cup s \cup s)_{i_0 \cdots i_6}
= - \delta({\gamma'}^{[0]} - \gamma^{[0]})_{i_0 \cdots i_6} 
$$
lead to the formula for $\check{\gamma} - \check{\gamma}$.
\end{proof}

Now, we suppose $[g] \in \mathrm{Ker}\mu_3$ and there is $\alpha : M \to \R$ such that $\alpha_i = \alpha|_{U_i}$. Also, we choose $m$ and $\lambda$ satisfying $\beta = m \cup h + D \lambda$ to define the \v{C}ech-de Rham $5$-cocycle $\omega$ as in Definition \ref{dfn:Cech_deRham_5_cochain}, where $\beta$ is defined as in Lemma \ref{lem:Cech_deRham_3_cocycle}. Let $\alpha' : M \to \R$ be another choice such that $\alpha_i = \alpha|_{U_i}$, and $\beta'$ the corresponding \v{C}ech-de Rham $3$-cocycle. If we put $\lambda' = \lambda$ and $m' = m - s$, then we have $\beta' = m' \cup h + D \lambda'$, where the $0$-cocycle $(s) \in Z^0(\{ U_i \}, \Z)$ is defined by the difference $s = \alpha' - \alpha$. Let $\omega'$ be the cocycle defined as in Definition \ref{dfn:Cech_deRham_5_cochain} by using $\beta', \alpha', m'$ and $\lambda'$. Then Lemma \ref{lem:change_alpha_in_beta} shows: $
\omega' = \omega$.

\subsection{The change of $\eta^{[0]}_i$}
\label{subsec:change_eta}

Let ${\eta'}^{[0]}_{ijk}$ be another choice of $\eta^{[0]}_{ijk}$ in Definition \ref{dfn:Deligne_4_cochain}. This choice defines $(h'_{ijkl}) \in Z^3(\{ U_i \}, \Z)$ by $h'_{ijkl} = \delta ({\eta'}^{[0]})_{ijkl}$. If we put $r_{ijk} = {\eta'}^{[0]}_{ijk} - \eta^{[0]}_{ijk}$ to express the difference of the choices, then $\check{\rho} = (r_{ijk}, 0) \in \check{C}^2(\{ U_i \}, \Z(1)_D^\infty)$ satisfies
$$
(h'_{ijkl}, {\eta'}^{[0]}_{ijk})
- (h_{ijkl}, \eta^{[0]}_{ijk})
= ( (\delta r)_{ijkl}, r_{ijk}) = D \check{\rho}.
$$
Let $\check{\alpha} = (a, \alpha^{[0]}) \in Z^1(\{ U_i \}, \Z(1)_D^\infty)$ be as in Subsection \ref{subsec:mu_1}. We regard $\check{\rho} \cup \check{\alpha}$ as a cochain belonging to $C^3(\{ U_i \}, Z_D^\infty(4))$.

\begin{lem} \label{lem:change_eta_in_beta}
Let $\check{\beta'}$ be the $4$-cocycle defined as in Definition \ref{dfn:Deligne_4_cochain} under the choices of ${\eta'}^{[0]}_{ijk}$ and $h'_{ijkl}$ above with the other choices unchanged. Then we have
$$
\check{\beta}' - \check{\beta}
= - D (\check{\rho} \cup \check{\alpha}).
$$
\end{lem}

\begin{proof}
The formula in question is equivalent to:
\begin{align*}
\check{\beta}' - \check{\beta}
&= (b' - b, {\beta'}^{[0]} - \beta^{[0]}, 
{\beta'}^{[1]} - \beta^{[1]}, 0, 0) \\
&= ( \delta (- r \cup a), 
\delta (- r \cup \alpha^{[0]}) + r \cup a, d(- r \cup \alpha^{[0]}), 0, 0) \\
&= D( - r \cup a, - r \cup \alpha^{[0]}, 0, 0),
\end{align*}
which can be readily verified.
\end{proof}

We set $\check{h}' = (h') \in Z^3(\{ U_i \}, \Z(0)_D^\infty)$, and regard the cup products $\check{h}' \cup \check{\rho} \cup \check{\alpha}$ and $\check{\rho} \cup \check{\beta}$ as cochains belonging to $C^6(\{ U_i \}, \Z(6)_D^\infty)$ in the following lemma.

\begin{lem}
Let $\check{\gamma}'$ be the $6$-cochain defined as in Definition \ref{dfn:Deligne_6_cochain} in the same way as $\check{\beta}'$. Then we have
\begin{align*}
\check{\gamma}' - \check{\gamma}
+ 2(\check{\rho} \cup \check{\beta}
+ \check{h}' \cup \check{\rho} \cup \check{\alpha})
= D (k, \kappa^{[0]}, \kappa^{[1]}, 0, 0, 0),
\end{align*}
where $(k, \kappa^{[0]}, \kappa^{[1]}, 0, 0, 0) \in C^5(\{ U_i \}, \Z(6)_D^\infty)$ is given by
\begin{align*}
\kappa^{[1]}_{0123}
&=
r_{012} \beta^{[1]}_{023} - r_{123} \beta^{[1]}_{013}, \\
\kappa^{[0]}_{01234}
&=
r_{012}r_{234} \alpha^{[0]}_4 
- h'_{1234} r_{014} \alpha^{[0]}_4 - h'_{0123} r_{034} \alpha^{[0]}_4 \\
& \quad
+ r_{234} \beta^{[0]}_{0124} 
- r_{123} \beta^{[0]}_{0134} + r_{012} \beta^{[0]}_{0234}, \\
k_{012345}
&=
- (h'_{1234} r_{014} + h'_{0123} r_{034} - r_{012} r_{234}) a_{45}
- h_{0123} r_{345} a_{35} \\
& \quad
- r_{345} b_{01235} + r_{234} b_{01245}
- r_{123} b_{01345} + r_{012} b_{02345}.
\end{align*}
In the above, we put $i_0 = 0, i_1 = 1, \ldots$ to suppress notations.
\end{lem}

\begin{proof}
The claim is equivalent to ${\gamma'}^{[5]} = \gamma^{[5]}$, ${\gamma'}^{[4]} = \gamma^{[4]}$ and
\begin{align*}
{\gamma'}^{[3]} - \gamma^{[3]}
+ 2 r \cup \beta^{[3]}
&= 0, \\
{\gamma'}^{[2]} - \gamma^{[2]}
+ 2 r \cup \beta^{[2]}
&= - d \kappa^{[1]}, \\
{\gamma'}^{[1]} - \gamma^{[1]}
+ 2 r \cup \beta^{[1]}
&= \delta \kappa^{[1]} + d \kappa^{[0]}, \\
{\gamma'}^{[0]} - \gamma^{[0]}
+ 2 (r \cup \beta^{[0]} + h' \cup r \cup \alpha^{[0]})
&= \delta \kappa^{[0]} - k, \\
c' - c 
+ 2( r \cup b + h' \cup r \cup a)
&= \delta k.
\end{align*}
The formulae ${\gamma'}^{[5]} = \gamma^{[5]}, \cdots, {\gamma'}^{[2]} - \gamma^{[2]} + 2 r \cup \beta^{[2]} = - d \kappa^{[1]}$ are rather easy. To prove ${\gamma'}^{[1]} - \gamma^{[1]} + 2 r \cup \beta^{[1]} - \delta \kappa^{[1]} = d \kappa^{[0]}$, we adapt the formula of $(\delta k)_{01234}$ in the proof of Lemma \ref{lem:coboundary_gamma} to the present case. For ${\gamma'}^{[0]} - \gamma^{[0]} + 2 (r \cup \beta^{[0]} + h' \cup r \cup \alpha^{[0]}) - \delta \kappa^{[0]} = - k$, we also adapt the formula of $(\delta \ell)_{012345}$ in the proof of Lemma \ref{lem:coboundary_gamma} and the following formula: If we put 
$$
m_{01234} 
= r_{234} \beta^{[0]}_{0124} 
- r_{123} \beta^{[0]}_{0134} 
+ r_{012} \beta^{[0]}_{0234},
$$
then we have
\begin{align*}
& 
(\delta m)_{012345}
=
r_{012} \beta^{[0]}_{2345} - r_{345} \beta^{[0]}_{0123} \\
& \
+ (\delta r)_{2345} \beta^{[0]}_{0125}
+ (\delta r)_{1234} \beta^{[0]}_{0145}
+ (\delta r)_{0123} \beta^{[0]}_{0345} \\
& \quad
+ r_{345} (\delta \beta^{[0]})_{01235}
- r_{234} (\delta \beta^{[0]})_{01245}
+ r_{123} (\delta \beta^{[0]})_{01345}
- r_{012} (\delta \beta^{[0]})_{02345}.
\end{align*}
To prove the remaining formula, we use three formulae: The first formula is that of $(\delta \ell)_{012345}$ in the proof of Lemma \ref{lem:coboundary_gamma} adapted to the present case. The second formula is the coboundary of the $2$-cochain $p_{012} = r_{012} a_{02}$:
$$
(\delta p)_{0123}
= 
- r_{123} a_{01} + r_{012} a_{123}
+ (\delta r)_{0123} a_{03}
+ r_{123} (\delta a)_{013} - r_{012} (\delta a)_{023}.
$$
The third formula is that the coboundary of the $5$-cochain 
$$
q_{012345}
= r_{345}b_{01235} 
- r_{234}b_{01245}
+ r_{123}b_{01345}
- r_{012}b_{02345}
$$
is given by
\begin{align*}
(\delta q)_{0123456}
&=
r_{456} b_{01234} - r_{012}b_{23456} \\
& 
- (\delta r)_{3456} b_{01236}
- (\delta r)_{2345} b_{01256}
- (\delta r)_{1234} b_{01456}
- (\delta r)_{0123} b_{03456} \\
& 
+ r_{456} (\delta b)_{012346}
- r_{345} (\delta b)_{012356}
+ r_{234} (\delta b)_{012456} \\
& 
- r_{123} (\delta b)_{013456}
+ r_{012} (\delta b)_{023456}.
\end{align*}
We can then compute $(\delta k)_{0123456}$ to verify the formula.
\end{proof}

Suppose $[g] \in \mathrm{Ker}\mu_3$, so that we can choose $\alpha$ such that $\alpha|_{U_i} = \alpha_i$ as well as $m$ and $\lambda$ satisfying $\beta = m \cup h + D \lambda$ to define $\omega$ as in Definition \ref{dfn:Cech_deRham_5_cochain}, where $\beta$ is defined as in Lemma \ref{lem:Cech_deRham_3_cocycle}. Let ${\eta'}^{[0]}_{ijk}$ be as in this subsection, and $\beta'$ the corresponding \v{C}ech-de Rham $3$-cocycle. If we define $m' = m$ and $\lambda' = (\lambda^{[0]}_{ijk} - r_{012}\alpha - m r_{012}, \lambda^{[1]}, \lambda^{[2]})$, then we have $\beta' = m' \cup h' + D \lambda'$. Let $\omega'$ be the \v{C}ech-de Rham $5$-cocycle defined by using $\beta'$, $\alpha$, $m'$ and $\lambda'$. Then we have
\begin{multline*}
\omega' - \omega \\
=
D(\kappa^{[0]} - 2r \cup \lambda^{[0]} + m (r \cup r - T - S),
\kappa^{[1]} - 2 r \cup \lambda^{[1]}, 
-2 r \cup \lambda^{[2]}, 
0, 0).
\end{multline*}
In the above, $(S_{i_0i_1i_2i_3i_4}) \in C^4(\{ U_i \}, \Z)$ and $(T_{i_0i_1i_2i_3i_4}) \in C^4(\{ U_i \}, \Z)$ are
\begin{align*}
S_{01234}
&= h'_{1234} r_{014} + h'_{0123} r_{034}, \\
T_{01234}
&= r_{234} h_{0124} - r_{123} h_{0134} + r_{012} h_{0234},
\end{align*}
where $i_0 = 0, i_1 = 1, \ldots$ to suppress notations. The formula of the difference $\omega' - \omega$ above can be shown by using lemmas in this subsection and
$$
Q(h')_{012345} - Q(h)_{012345}
= \delta (T + S - r \cup r)_{012345} 
- 2 r_{012} h_{2345} + 2 h'_{0123} r_{345}.
$$

\subsection{The change of $\phi_{ij}$}
\label{subsec:change_phi}

If $\phi'_{ij} : U_i \to U(H)$ is another choice of a lift of the transition function $\bar{\phi}_{ij}$, then there is $\rho^{[0]}_{ij} : U_{ij} \to \R$ such that $\phi_{ij} = \phi'_{ij} \exp 2\pi i \rho^{[0]}_{ij}$. In this case, we can choose ${\eta'}^{[0]}_{ijk}$ to be ${\eta'}^{[0]}_{ijk} = \eta^{[0]}_{ijk} + (\delta \rho^{[0]})_{ijk}$, so that $h'_{ijkl} = h_{ijkl}$. Hence the cochain $(0, \rho^{[0]}_{ij}) \in \check{C}^2(\{ U_i \}, \Z(1)_D^\infty)$ satisfies
$$
(h'_{ijkl}, {\eta'}^{[0]}_{ijk})
- (h_{ijkl}, \eta^{[0]}_{ijk})
= ( 0, (\delta \rho^{[0]})_{ijk}) = D  (0, \rho^{[0]}_{ij}).
$$

\begin{lem} \label{lem:change_phi_in_beta}
Let $\check{\beta'}$ be the $4$-cocycle defined as in Definition \ref{dfn:Deligne_4_cochain} under the choices of $\phi'_{ij}$ and ${\eta'}^{[0]}_{ijk}$ above with the other choices unchanged. Then we have
$$
\check{\beta}' - \check{\beta}
= (0, 0, \delta(- \rho^{[0]} d\alpha^{[0]})_{ijk}, 
d( \rho^{[0]} d\alpha^{[0]})_{ij}, 0)
= D(0, 0, - \rho^{[0]}_{ij} d\alpha^{[0]}_j, 0)
$$
where $\check{\alpha} = (a, \alpha^{[0]}) \in Z^1(\{ U_i \}, \Z(1)_D^\infty)$ is as in Subsection \ref{subsec:mu_1}, and $\check{\rho} \cup \check{\alpha}$ is regarded as a cochain in $C^3(\{ U_i \}, Z_D^\infty(4))$. 
\end{lem}

\begin{proof}
Clearly, $b' = b$, ${\beta'}^{[0]} = \beta^{[0]}$ and ${\beta'}^{[3]} = \beta^{[3]}$. The formulae
\begin{align*}
{\beta'}^{[1]}_{ijk} - \beta^{[1]}_{ijk} 
&= \delta(- \rho^{[0]} d\alpha^{[0]})_{ijk}, &
{\beta'}^{[2]}_{ij} - \beta^{[2]}_{ij} 
&= d( \rho^{[0]}_{ij} d\alpha^{[0]}_j),
\end{align*}
can be verified straightly by computations.
\end{proof}

\begin{lem}
Let $\check{\gamma}'$ be the $6$-cochain defined as in Definition \ref{dfn:Deligne_6_cochain} in the same way as $\check{\beta}'$. Then we have
$$
\check{\gamma}' - \check{\gamma}
+ 2 (0, 0, h \cup \rho \cup d\alpha^{[0]}, 0, 0, 0, 0)
= - D (0, 0, \zeta^{[1]}, \zeta^{[2]}, \zeta^{[3]}, 0),
$$
where $(0, 0, \zeta^{[1]}, \zeta^{[2]}, \zeta^{[3]}, 0) \in C^5(\{ U_i \}, \Z(6)_D^\infty)$ is given by
\begin{align*}
\zeta^{[3]}_{01}
&=
2 \rho^{[0]}_{01} \beta^{[3]}_1 + d\rho^{[0]}_{01} \beta^{[2]}_{01} 
- \frac{1}{24\pi^2} d\rho^{[0]}_{01} 
\{ B_2(g_1, \phi_{10}) + B_2(\phi_{10}, g_0) \}, \\
\zeta^{[2]}_{012}
&=
2 \rho^{[0]}_{01} \beta^{[2]}_{12} 
+ \rho^{[0]}_{01} d\rho^{[0]}_{12} d\alpha^{[0]}_2 \\
& \quad
+ (\delta \rho^{[0]})_{012} d \beta^{[1]}_{012} 
- 2d\rho^{[0]}_{02} \beta^{[1]}_{012}
+ (\delta \rho^{[0]})_{012} d\rho^{[0]}_{02} d\alpha^{[0]}_2, \\
\zeta^{[1]}_{0123}
&=
2 \rho^{[0]}_{01} \beta^{[1]}_{123} 
- \rho^{[0]}_{01} (\delta \rho^{[0]})_{123} d\alpha^{[0]}_3 \\
& \quad
- 2 \rho^{[0]}_{03} d\beta^{[0]}_{0123} 
+ (\delta \rho^{[0]})_{023} \beta^{[1]}_{012}
- (\delta \rho^{[0]})_{013} \beta^{[1]}_{123}.
\end{align*}
In the above, we put $i_0 = 0, i_1 = 1, \ldots$ to suppress notations.
\end{lem}

\begin{proof}
The lemma amounts to the following formulae:
\begin{align*}
{\gamma'}^{[5]} - \gamma^{[5]}
&= 0, \\
{\gamma'}^{[4]} - \gamma^{[4]}
&= d \zeta^{[3]}, \\
{\gamma'}^{[3]} - \gamma^{[3]}
&= - \delta \zeta^{[3]} - d \zeta^{[2]}, \\
{\gamma'}^{[2]} - \gamma^{[2]}
&= - \delta \zeta^{[2]} + d \zeta^{[0]}, \\
{\gamma'}^{[1]} - \gamma^{[1]}
- 2 h \cup \rho^{[0]} \cup d\alpha^{[0]}
&= - \delta \zeta^{[1]}, \\
c' - c 
&= 0.
\end{align*}
The first formula ${\gamma'}^{[5]} = \gamma^{[5]}$ is clear, and the second formula is easy to check. The computation verifying the third formula is lengthy: A comment about the computation is that, among the various terms in ${\gamma'}^{[3]} - \gamma^{[3]} + \delta \zeta^{[3]}$, the terms involving $B_2$ explicitly give $d\rho^{[0]}_{12} \beta^{[2]}_{01} - d \rho^{[0]}_{01} \beta^{[2]}_{12}$ after applications of the cocycle condition for $B_2$. In the verification of the fourth formula, useful is the formula of the coboundary of the $2$-cochain $p_{012} = d\rho^{[0]}_{02} \beta^{[1]}_{012}$:
\begin{align*}
(\delta p)_{0123}
&= - d\rho^{[0]}_{01} \beta^{[1]}_{123} + d\rho^{[0]}_{23} \beta^{[1]}_{012}
+ d\rho^{[0]}_{03} (\delta \beta^{[1]})_{0123} \\
& \quad
+ (\delta d\rho^{[0]})_{013} \beta^{[1]}_{123} 
- (\delta d\rho^{[0]})_{023} \beta^{[1]}_{012}.
\end{align*}
This gives a similar formula for the coboundary of the $2$-cochain $(\delta \rho^{[0]})_{012} d\rho^{[0]}_{02}$. In the verification of the fifth formula, we use the formula of the coboundary of the $3$-cochain $q_{0123} = \rho^{[0]}_{03} d\beta^{[0]}_{0123}$:
\begin{align*}
(\delta q)_{01234}
&=
- \rho^{[0]}_{01} d\beta^{[0]}_{1234} - \rho^{[0]}_{34} d\beta^{[0]}_{0123}
+ \rho^{[0]}_{04} (\delta d\beta^{[0]})_{01234} \\
& \quad
+ (\delta \rho^{[0]})_{014} d\beta^{[0]}_{1234} 
+ (\delta \rho^{[0]})_{034} d\beta^{[0]}_{0123}.
\end{align*}
The formula of the coboundary of $k_{0123} = (\delta \rho^{[0]})_{023} \beta^{[1]}_{012} - (\delta \rho^{[0]})_{013} \beta^{[1]}_{123}$ is also useful, which can be derived from a formula presented in the proof of Lemma \ref{lem:coboundary_gamma}. Finally, $c' = c$ is apparent.
\end{proof}

Suppose $[g] \in \mathrm{Ker}\mu_3$. Accordingly, we choose $\alpha$ such that $\alpha|_{U_i} = \alpha_i$, and $m$ and $\lambda$ satisfying $\beta = m \cup h + D \lambda$ to define $\omega$ as in Definition \ref{dfn:Cech_deRham_5_cochain}, where $\beta$ is defined as in Lemma \ref{lem:Cech_deRham_3_cocycle}. Let $\phi'_{ij}$ be as in this subsection, and $\beta'$ the corresponding \v{C}ech-de Rham $3$-cocycle. If we put $m' = m$ and $\lambda' = (\lambda^{[0]}, \lambda^{[1]} - \rho^{[0]} d\alpha, \lambda^{[2]})$, then $\beta' = m' \cup h' + D \lambda'$. Let $\omega'$ be the \v{C}ech-de Rham $5$-cocycle defined by using $\beta'$, $\alpha$, $m'$ and $\lambda'$. Then the lemmas in this subsection prove
$$
\omega' - \omega
=
- D(0, \zeta^{[1]}, \zeta^{[2]}, \zeta^{[3]}, 0).
$$

\subsection{The change of $s_i$}
\label{subsec:change_s}

We chose another local section $s'_i$ of $P$ in Definition \ref{dfn:Deligne_4_cochain}. Then there exists $\bar{\psi}_i : U_i \to PU(H)$ such that $s_i = s'_i \bar{\psi}_i$. The induced transition functions $\{ \bar{\phi}_{ij} \}$ and $\{ \bar{\phi}'_{ij} \}$ are related by $\bar{\phi}'_{ij} = \bar{\psi}_i \bar{\phi}_{ij} \bar{\psi}_j^{-1}$. We choose a lift $\psi_i : U_i \to U(H)$ of $\bar{\psi}_i$. From a lift $\phi_{ij}$ of $\bar{\phi}_{ij}$, we can construct a lift $\phi'_{ij}$ of $\bar{\phi}'_{ij}$ by setting $\phi'_{ij} = \psi_i \phi_{ij} \psi_j^{-1}$, which gives $f'_{ijk} = f_{ijk}$ as well as $h'_{ijkl} = h_{ijkl}$.

\begin{lem} \label{lem:change_s_in_beta}
Let $\check{\beta'}$ be the $4$-cocycle defined as in Definition \ref{dfn:Deligne_4_cochain} under the choices of $\phi'_{ij}$ above with the other choices unchanged. Then we have
$$
\check{\beta'} - \check{\beta}
= (0, 0, 0, \delta \tau^{[2]}_{ij}, d \tau^{[2]}_i)
= D(0, 0, 0, \tau^{[2]}_i),
$$
where $\tau^{[2]}_i \in \Omega^2(U_i)$ is given by
$$
\tau^{[2]}_i
=
\frac{1}{8\pi^2}
\{ B(\psi_i, g_i) + B(g'_i, \psi_i) \}.
$$
\end{lem}

\begin{proof}
It is enough to prove the formulae:
\begin{align*}
{\beta'}^{[3]}_i - \beta^{[3]}_i
&= d \tau_i^{[2]}, &
{\beta'}^{[2]}_{ij} - \beta^{[3]}_{ij}
&= (\delta \tau^{[2]})_{ij}.
\end{align*}
The first formula follows from Lemma \ref{lem:conjugate_C3}. To prove the second formula, we substitute $i = 0$ and $j = 1$ for simplicity. We then get
\begin{align*}
 &
8 \pi^2
\{ 
{\beta'}^{[2]}_{ij} - \beta^{[2]}_{ij} - (\delta \tau^{[2]})_{ij}
\} \\
& \quad
=
\delta B(g'_1, \phi'_{10}, \psi_0)
- \delta B(\phi'_{10}, g'_0, \psi_0)
+ \delta B(\phi'_{10}, \psi_0, g_0) \\
& \quad
- \delta B(g'_1, \psi_1, \phi_{10})
+ \delta B(\psi_1, g_1, \phi_{10})
- \delta B(\psi_1, \phi_{10}, g_0) 
=0,
\end{align*}
by using the cocycle condition for $B = B_2$.
\end{proof}

\begin{lem}
Let $\check{\gamma}'$ be the $6$-cochain defined as in Definition \ref{dfn:Deligne_6_cochain} in the same way as $\check{\beta}'$. Then we have
\begin{align*}
\check{\gamma}' - \check{\gamma}
-
(0, 0, 0, 2h \cup \tau^{[2]}, 0, 0, 0)
= D (0, 0, 0, \xi^{[2]}, \xi^{[3]}, \xi^{[4]}),
\end{align*}
where $(0, 0, 0, \xi^{[2]}, \xi^{[3]}, \xi^{[4]}) \in C^5(\{ U_i \}, \Z(6)_D^\infty)$ is given by
\begin{align*}
\xi^{[4]}_{0}
&=
\frac{i}{48\pi^3}
\{ B_4(g'_0, \psi_0) - B_4(\psi_0, g_0) \}, \\
\xi^{[3]}_{01}
&=
\frac{i}{48\pi^3}
\{ 
A(g'_1, \phi'_{10}, \psi_0)
- A(\phi'_{10}, g'_0, \psi_0)
+ A(\phi'_{10}, \psi_0, g_0)
\} \\
& 
+ \frac{i}{48\pi^3}
\{
- A(g'_1, \psi_1, \phi_{10})
+ A(\psi_1, g_1, \phi_{10})
- A(\psi_1, \phi_{10}, g_0)
\}, \\
\xi^{[2]}_{012}
&=
- 2 \eta^{[0]}_{012} \tau^{[2]}_2.
\end{align*}
In the above, we put $i_0 = 0, i_1 = 1, \ldots$ to suppress notations.
\end{lem}

\begin{proof}
The formulae to be shown are as follows:
\begin{align*}
{\gamma'}^{[5]} - \gamma^{[5]} 
&= d \xi^{[4]}, \\
{\gamma'}^{[4]} - \gamma^{[4]} 
&= (\delta \xi^{[4]}) - d \xi^{[3]}, \\
{\gamma'}^{[3]} - \gamma^{[3]} 
&= (\delta \xi^{[3]}) + d \xi^{[2]}, \\
{\gamma'}^{[2]} - \gamma^{[2]} 
- 2h \cup \tau^{[2]}
&= (\delta \xi^{[3]}) + d \xi^{[2]}, \\
{\gamma'}^{[1]} - \gamma^{[1]} 
&= 0, \\
{\gamma'}^{[0]} - \gamma^{[0]} 
&= 0, \\
c' - c
&= 0.
\end{align*}
The first formula follows from Lemma \ref{lem:conjugate_C5}. The second formula follows from the formula $\delta B_4 = d A$ in Lemma \ref{lem:key_formulae}. For the third formula, we use the identity
\begin{align*}
0
&
= \{ \delta A(g'_2, \phi'_{21}, \phi'_{10}, \psi_0)
- \delta A(\phi'_{21}, g'_1, \phi'_{10}, \psi_0)
+ \delta A(\phi'_{21}, \phi'_{10}, g'_0, \psi_0) \\
& 
- \delta A(\phi'_{21}, \phi'_{10}, \psi_0, g_0) \}
- \{
\delta A(g'_2, \phi'_{21}, \psi_1, \phi_{10})
- \delta A(\phi'_{21}, g'_1, \psi_1, \phi_{10}) \\
& 
+ \delta A(\phi'_{21}, \psi_1, g_1, \phi_{10})
- \delta A(\phi'_{21}, \psi_1, \phi_{10}, g_0) \} 
+ \{ \delta A(g'_2, \psi_2, \phi_{21}, \phi_{10}) \\
&
- \delta A(\psi_2, g_2, \phi_{21}, \phi_{10})
+ \delta A(\psi_2, \phi_{21}, g_1, \phi_{10})
- \delta A(\psi_2, \phi_{21}, \phi_{10}, g_0)
\},
\end{align*}
and Lemma \ref{lem:change_s_in_beta}. The remaining formulae are easy to verify.
\end{proof}

Suppose $[g] \in \mathrm{Ker}\mu_3$ as in Definition \ref{dfn:Cech_deRham_5_cochain}. Hence we choose $\alpha$ such that $\alpha|_{U_i} = \alpha_i$, and also $m$ and $\lambda$ satisfying $\beta = m \cup h + D \lambda$ to define $\omega$, where $\beta$ is defined as in Lemma \ref{lem:Cech_deRham_3_cocycle}. Let $g'_i$ and $\phi'_{ij}$ be as in this subsection, and $\beta'$ the corresponding \v{C}ech-de Rham $3$-cocycle. If we put $m' = m$ and $\lambda' = (\lambda^{[0]}, \lambda^{[1]}, \lambda^{[2]} + \tau^{[2]})$, then $\beta' = m' \cup h' + D \lambda'$. Let $\omega'$ be the \v{C}ech-de Rham $5$-cocycle defined by using $\beta'$, $\alpha$, $m'$ and $\lambda'$. The lemmas in this subsection then establish
$$
\omega' - \omega
=
 D(0, 0, \xi^{[2]}, \xi^{[3]}, \xi^{[4]}).
$$


Department of Mathematical Sciences,
Shinshu University, 
3--1--1 Asahi, Matsumoto, Nagano 390-8621,
Japan.

{\it E-mail address}: \verb|kgomi@math.shinshu-u.ac.jp|

\end{document}